\documentclass[11pt]{article}
\date{}

\usepackage[title]{appendix}
\usepackage{xcolor}
\usepackage[margin=1in]{geometry}                
\usepackage{graphicx}
\usepackage{subcaption}
\usepackage{pdflscape}
\usepackage{amssymb}
\usepackage[normalem]{ulem}
\usepackage{hyperref}
\usepackage{enumitem}
\usepackage{mathrsfs}
\usepackage{epstopdf}
\usepackage{rotating}
\usepackage{longtable} 
\usepackage{adjustbox}
\usepackage{float}

\usepackage{color}
\usepackage{bbm, dsfont}
\usepackage{pst-node}
\usepackage{tikz-cd}
\usepackage{amsfonts} 
\usepackage{geometry}
\usepackage{amsthm}
\usepackage{amsmath}
\usepackage{titlesec}
\usepackage{amssymb}
\usepackage{enumitem}
\usepackage{float}
\usepackage [english]{babel}
\usepackage [autostyle, english = american]{csquotes}
\usepackage{algorithm}
\usepackage[noend]{algpseudocode} 
\makeatletter
\def\BState{\State\hskip-\ALG@thistlm}
\makeatother
\usepackage{hyperref}
\usepackage[normalem]{ulem}
\usepackage{mathrsfs}
\usepackage[italicdiff]{physics}

\newlist{casess}{enumerate}{1}
\setlist[casess]{label=     \textbf{Case} \arabic*:}
\usepackage{mathtools}

\DeclarePairedDelimiter\floor{\lfloor}{\rfloor}

\makeatletter
\newcommand*{\rom}[1]{\expandafter\@slowromancap\romannumeral #1@}
\makeatother

\usepackage{etoolbox}

\makeatletter
\patchcmd{\ttlh@hang}{\parindent\z@}{\parindent\z@\leavevmode}{}{}
\patchcmd{\ttlh@hang}{\noindent}{}{}{}
\makeatother

\usepackage{listings}
\usepackage{color} 
\definecolor{mygreen}{RGB}{28,172,0} 
\definecolor{mylilas}{RGB}{170,55,241}

\newlist{Assumptions}{enumerate}{1}
\setlist[Assumptions]{label=     \textbf{Assumption} \arabic*:}

\makeatletter

\newsavebox{\@brx}
\newcommand{\llangle}[1][]{\savebox{\@brx}{\(\m@th{#1\langle}\)}%
  \mathopen{\copy\@brx\kern-0.5\wd\@brx\usebox{\@brx}}}
\newcommand{\rrangle}[1][]{\savebox{\@brx}{\(\m@th{#1\rangle}\)}%
  \mathclose{\copy\@brx\kern-0.5\wd\@brx\usebox{\@brx}}}
\makeatother

\usepackage{lipsum} 
\usepackage{titlesec}
\titleformat{\subsection}[runin]
       {\normalfont\bfseries}
       {\thesubsection}
       {0.5em}
       {}
       [.]

 \newtheorem{thm}{Theorem}[section]
 \newtheorem{cor}[thm]{Corollary}
 
 \newtheorem{lem}[thm]{Lemma}
 \newtheorem{prop}[thm]{Proposition}
 \theoremstyle{definition}
 \newtheorem{defn}[thm]{Definition}
 \theoremstyle{remark}
 \newtheorem{rem}[thm]{Remark}
 \newtheorem{ass}[thm]{Assumption}
 \newtheorem{ex}[thm]{Example}
 
 \numberwithin{equation}{section}

\numberwithin{equation}{section}

\def\bu{\bullet}
\newcommand{\A}{\mathfrak{A}}
\newcommand{\B}{\mathfrak{B}}

\newcommand{\X}{\mathcal{X}} 
\newcommand{\C}{\mathbb{C}}

\def\BofH{\mathbb B(\mathcal H)}

\def\N{\mathbb{N}}
\def\T{\mathbb{T}}
\def\Z{\mathbb{Z}}

\def\H{\mathcal H}

\def\R{\mathbb{R}}
\def\Z{\mathbb Z}

\def\e{{\sf e}}

\DeclareMathOperator{\Span}{Span}
\DeclarePairedDelimiterX{\inp}[2]{\langle}{\rangle}{#1, #2}

\makeatletter
\newcommand*\bigcdot{\mathpalette\bigcdot@{.5}}
\newcommand*\bigcdot@[2]{\mathbin{\vcenter{\hbox{\scalebox{#2}{$\m@th#1\bullet$}}}}}
\makeatother

\def\d{{\rm d}}

\def\CC{\mathbb C}

\def\<{\langle}
\def\>{\rangle}
\providecommand{\norm}[1]{\lVert#1\rVert}

\numberwithin{equation}{section}

\usepackage[backend=biber,maxnames=10]{biblatex}
\begin{document}

\title{Harmonic analysis and automatic continuity in the context of generalized differential subalgebras}

\author{Felipe I. Flores
\footnote{
\textbf{2020 Mathematics Subject Classification:} Primary 46K05, Secondary 46H30, 46H40.
\newline
\textbf{Key Words:} generalized differential subalgebra, abstract harmonic analysis, functional calculus, minimal ideals, Wiener property, automatic continuity.}
}

\maketitle

\begin{abstract}\setlength{\parindent}{0pt}\setlength{\parskip}{1ex}\noindent
 For appropriate parameters $k,p,q$, we introduce and systematically study the class of $(k,p,q)$-differential subalgebras. This is a vast class of Banach $^*$-algebras defined by their relation with their $C^*$-envelopes. Some examples are given by normable two-sided $^*$-ideals, domains of closed $^*$-derivations, full Hilbert algebras, and some weighted convolution algebras of various kinds. We prove that this class of algebras possesses various interesting properties, such as closedness under a functional calculus based on smooth functions, $^*$-regularity, Wiener's property $(W)$, and properties of automatic continuity. 
\end{abstract}

\tableofcontents

\section{Introduction}

Let $(\A,\norm{\cdot}_{\A})$, $(\B,\norm{\cdot}_{\B})$ be Banach $^*$-algebras, such that $\A$ is a Banach $^*$-subalgebra of $\B$. We say that $\A$ is a {\it differential subalgebra} of $\B$ if there is a constant $C>0$ such that 
\begin{equation*}
    \norm{ab}_{\A}\leq C(\norm{a}_{\A}\norm{b}_{\B}+\norm{a}_{\B}\norm{b}_{\A}),\quad \textup{ for all } a,b\in\A.
\end{equation*}

The concept of differential subalgebras appeared first in the articles \cite{BC91,KS94}. This property gained some popularity after its appearance in the articles by Gr\"ochenig and Klotz \cite{GK13,GK14}, where it was used to provide norm bounds on the holomorphic functional calculus of the algebra. More concretely, they prove that, if $\A$ is a unital differential subalgebra of a $C^*$-algebra $\B$, then $\A$ admits \emph{norm-controlled inversion} in $\B$. This property can be viewed as a quantitative version of Wiener's Tauberian theorem (see \cite{Wi32}), as we now explain. Namely, $\A$ admits norm-controlled inversion in $\B$ if whenever an element $a\in \A$ is invertible in $\B$, then $a^{-1}$ lies in $\A$ and, moreover, there exists a function $h:\R^+\times \R^+\to \R^+$ that controls the norm of $a^{-1}$, in the sense that 
$$
\norm{a^{-1}}_\A\leq h\big(\norm{a}_\A,\norm{a^{-1}}_\B\big).
$$

Following the breakthrough of Gr\"ochenig and Klotz, various authors have studied examples of inclusions $\A\subset \B$ satisfying the `generalized differential subalgebra condition of order $\theta$' \cite{SaSh19,FS20,ShSu20,FSSCT23}, greatly extending their results. More precisely, given the parameter $\theta\in(0,1)$ and algebras $\A,\B$ as above, the generalized differential subalgebra condition of order $\theta$ is satisfied if there exists a constant $C>0$ such that
\begin{equation}\label{gdiff}
    \norm{a^2}_{\A}\leq C\norm{a}_{\A}^{1-\theta}\norm{a}_{\B}^{1+\theta},\quad \textup{ for all } a\in\A.
\end{equation}
In the future and for the sake of simplicity, we will avoid referencing the parameter $\theta$ and simply refer to these algebras as `generalized differential subalgebras'.

Much of the work on generalized differential subalgebras has been motivated by the goal of obtaining a quantitative version of Wiener's Tauberian theorem. In this regard, we can cite the works of Sun \cite{Su05,Su07,Su08}, of Samei and Shepelska \cite{SaSh19}, of Shin and Sun \cite{ShSu20}, of Fang and Shin \cite{FS20}, and of Fang, Shen, Shin and Tao \cite{FSSCT23}. Many of these works focus on concrete algebras -often matrix algebras-, and have very concrete, practical applications in mind. For example, the works \cite{Su05,Su07} cite motivations coming from the theory of Gabor frames, non-uniform sampling, numerical analysis and the study of pseudodifferential operators.

On the more abstract side, generalized differential subalgebras have played a prominent role in Lafforgue's celebrated work on the rapid decay property and the Baum-Connes conjecture \cite{La00,La02}, and more recently in the study of the K-theory of $L^p$-operator algebras by Austad, Ortega and Palmstr\o m \cite{AOP25}.

However, recent works of Shin and Sun \cite{ShSu20} and of the author \cite{Fl24} have provided new examples of algebras satisfying more general conditions of a differential type. Both examples will be properly introduced in the next section, but let us remark now that the example in \cite{ShSu20} is a noncommutative algebra inspired by Beurling algebras over $\R^d$, whereas the example in \cite{Fl24}, here denoted $\mathfrak E$, is constructed as a weighted (Beurling type) subalgebra of a generalized convolution algebra. In any case, these algebras satisfy the following condition: \begin{equation}\label{4diff}
    \norm{a^4}_{\A}\leq C\norm{a}_{\A}^{p}\norm{a}_{{\rm C^*}(\A)}^{q},\quad \textup{ for all } a\in\mathfrak E,
\end{equation} 
where $p,q,C>0$ and $p+q=4$.

These findings mark the starting point of the present article, as they show that more general differential conditions still appear in nature. Moreover, the Beurling-type nature of these and previous examples supports the idea that the study of differential conditions can have an impact on harmonic analysis and automatic continuity. We understand that this point of view is new and, in order to pursue it, we aim to develop a theory of generalized differential subalgebras. The results within this theory not only recover many results that are already available, but we also provide, as applications, many new examples of interesting harmonic-analytic behavior. In fact, many of our results are already new in the setting of generalized differential subalgebras, or even in the more restrictive setting of differential subalgebras.

To further illustrate the last point, we mention that our results apply to Beurling algebras associated with groups of strong subexponential growth (see Subsection \ref{weighttt}), marking the first time that such results have been obtained for this class of groups. The specific nature of these results will be discussed in the next paragraphs.

Another set of results comes from the theory of automatic continuity. In particular, here we focus on the continuity of intertwining operators, as introduced by Bade and Curtis \cite{BD78}. Intertwining operators are a common generalization of algebra homomorphisms, derivations, and bimodule homomorphisms, making them a natural place to test for automatic continuity. The continuity of intertwining operators has been previously studied in \cite{BD78,Ru96,Ru98,Fl25}. 

In \cite{Fl25}, the author extended Runde's (\cite{Ru96}) criterion for the continuity of intertwining operators to the setting of generalized convolution algebras. In that article, the previously mentioned algebra $\mathfrak E$ played a central role. While those results did not fundamentally rely on $\mathfrak E$ satisfying a differential-type condition, they did depend on its harmonic-analytic structure. 

In this work, we take that observation as an indication that generalized differential subalgebras may lie at the heart of new developments in the theory of automatic continuity, both as algebras to study by themselves and as tools for studying automatic continuity in larger algebras containing them. With this in mind, we apply the harmonic-analytic techniques developed here to establish new automatic continuity results, in the spirit of \cite{Fl25,flores2025der}. We emphasize that this is the first time such results have been obtained for the class of differential subalgebras.

Let us now introduce our object of study, that is, the $(k,p,q)$-differential condition. Let $\A$ be a Banach $^*$-subalgebra of the $C^*$-algebra $\B$ such that the involution is continuous in $\A$. Given $k\in \Z^+\setminus\{1\}$ and $p,q>0$ such that $p+q=k$, we will say that $\A$ is a {\it $(k,p,q)$-differential subalgebra} of $\B$ if there exists a constant $C>0$ such that
\begin{equation*}
    \|{a^{k}}\|_{\A}\leq C\norm{a}_{\A}^p\norm{a}_{\B}^q,\quad \textup{ for all } a\in\A.
\end{equation*}
This condition clearly generalizes \eqref{gdiff} and \eqref{4diff}, making them particular examples of our study. In fact, the class of subalgebras satisfying such a condition is arguably very wide, and it includes examples of seemingly different natures. Concrete examples will be provided in the next section of the paper, which is very much devoted to such a goal, but we anticipate some of them here: normable two-sided $^*$-ideals, domains of unbounded $^*$-derivations, full Hilbert algebras, and some weighted convolution algebras of various kinds.

We remark now that one could study the same condition for general Banach algebras; there is no logical obstruction to study $\A$ and $\B$ just Banach algebras satisfying the identity above. However, we chose to keep $\B$ a $C^*$-algebra to take advantage of the rigidity of $C^*$-norms. Indeed, we hope to convince the reader that $(k,p,q)$-differential subalgebras are `$C^*$-like', and that we can derive for them many of the properties that $C^*$-algebras have. One could then say that our approach is deeply inspired by the theory of $C^*$-algebras.

The first result in this direction is the presence of a functional calculus based on smooth functions. All Banach algebras have a holomorphic functional calculus, whereas $C^*$-algebras have a functional calculus based on continuous functions. In fact, one could argue that this is the defining property that $C^*$-algebras have. It follows that, in some sense, the presence of a bigger functional calculus should suggest that the algebra is closer to a $C^*$-algebra.

In this article we will construct a functional calculus for $(k,p,q)$-differential subalgebras based on the Dixmier-Baillet construction \cite{Di60,MB79}. Other forms of functional calculi have been constructed in \cite{KS94,DLM04,Fl24,carter2025w}. In fact, Kissin and Shul'man constructed a functional calculus for differential subalgebras \cite{KS94} and our construction somehow generalizes theirs. No result like this was available for generalized differential subalgebras or beyond. 

The main difference with the functional calculus of Kissin and Shul'man is that their calculus involves all smooth functions, but, in our case and given the notable increase in generality, we can only guarantee the existence of the calculus for a smaller class of functions. We proceed to describe such functions now.

Given the real parameter $\tau\in(0,1)$, we consider the algebra $\mathfrak D_\tau$, consisting of functions that satisfy
$$
f(x)=\frac{1}{2\pi}\int_{\mathbb R} \widehat{f}(t)e^{itx}\d t,
$$ 
where $ \widehat{f}\in L^1(\R)$ and
$$
    \norm{f}_{\tau}=\int_{\R}|\widehat{f}(t)|e^{|t|^{\tau}}\d t<\infty.
$$ 
Under the norm $\norm{\cdot}_{\tau}$, $\mathfrak D_\tau$ is a Banach subalgebra of $C_0(\R)$. It is an old result of Domar that $\mathfrak D_\tau$ defines a regular commutative Banach algebra \cite{Do56}, meaning that this algebra is robust enough to separate points and closed sets in $\R$. This property makes a functional calculus based on $\mathfrak D_\tau$-functions very robust.

Our first main theorem precisely states the existence of such calculus. In it, we use $\widetilde{\A}$ to denote the minimal unitization of $\A$ and $\A(a)$ to denote the smallest closed subalgebra containing the element $a\in\A$.

\begin{thm}
    Let $\A$ be a $(k,p,q)$-differential subalgebra of the $C^*$-algebra $\B$. Then there exists a $\tau\in(0,1)$ such that, for every $a=a^*\in \A$, one has
    $$
    \norm{e^{ita}}_{\widetilde{\A}}=O(e^{|t|^\tau}),\quad\textup{ as }|t|\to\infty.
    $$ 
    Furthermore, the following are true for all $f\in\mathfrak D_\tau$.
    \begin{enumerate}
        \item[(i)] The following Bochner integral exists in $\widetilde{\A}(a,1)$: \begin{equation*}
            f(a)=\frac{1}{2\pi}\int_{\mathbb R} \widehat{f}(t)e^{ita}\d t.
        \end{equation*} Moreover, if $f(0)=0$, then $f(a)\in \A(a)$.
        \item[(ii)] For any non-degenerate $^*$-representation $\Pi:\widetilde{\A}(a,1)\to\BofH$, we have ${\Pi}\big(f(a)\big)=f\big(\Pi(a)\big)$. 
        \item[(iii)] ${\rm Spec}_{\widetilde{\A}(a,1)}\big(f(a)\big)=f\big({\rm Spec}_{\widetilde{\A}(a,1)}(a)\big)$.
        \item[(iv)] There is a (continuous) $^*$-homomorphism $\varphi_a:\mathfrak D_\tau\to \widetilde{\A}(a,1)$ defined by $\varphi_a(f)=f(a)$, such that $\varphi_a(g)=1$ if $g\equiv 1$ on a neighborhood of ${\rm Spec}_{\widetilde{\A}(a,1)}(a)$ and $\varphi_a(g)=a$ if $g(x)=x$ on a neighborhood of ${\rm Spec}_{\widetilde{\A}(a,1)}(a)$.
    \end{enumerate}
\end{thm}

Let us now describe the many applications this functional calculus has. We will use our functional calculus to study properties of spectral nature like symmetry, ideal-theoretic properties like $^*$-regularity, and representation-theoretic properties like the Wiener property. 

To be concrete, a reduced Banach $^*$-algebra is said to \begin{enumerate}
    \item[(i)] be symmetric if it is closed under the holomorphic functional calculus of its universal $C^*$-completion,
    \item[(ii)] be $^*$-regular if its $^*$-structure space is that of its universal $C^*$-completion (or, more precisely, the canonical map is an homeomorphism),
    \item[(iii)] have the Wiener property $(W)$ if every closed, proper ideal is annihilated by a topologically irreducible $^*$-representation.
\end{enumerate}

These properties have been widely studied in the context of group algebras and their variants. We mention, for example, the works \cite{Le73,Lu80,Ba81,Bo84,Ni99,Lo01,AuTh25} on group algebras. For weighted group algebras (Beurling algebras), see \cite{HKK83,MM98,FGLLM03,DLM04,SaSh19}. More recent generalizations involve the study of weighted $L^p$-algebras \cite{KuMB12} and even Orlicz algebras \cite{OzSa17,carter2025w}. In the setting of generalized convolution algebras, we can also mention the works \cite{Le75,Ki82,Le11,LeNg061,Fl24}. Let us remark that, while many of these articles are devoted to the study of one of the properties from above, it is also often the case that the three properties are studied together (see \cite{FGLLM03,DLM04,Fl24,carter2025w}).

Our next theorem states that $(k,p,q)$-differential subalgebras have all the properties we mentioned above.

\begin{thm}\label{main2}
    Let $\A$ be a $(k,p,q)$-differential subalgebra of the $C^*$-algebra $\B$. Then $\A$ is symmetric, $^*$-regular and has the Wiener property $(W)$.
\end{thm}

We should now mention that symmetry and $^*$-regularity have been established for differential subalgebras \cite{KS94,GK13} and that symmetry is also known for generalized differential subalgebras \cite{SaSh19}. The fact that the Wiener property $(W)$ holds is new at all levels of generality.

In terms of applications, this theorem immediately recovers many pre-existing results, in a unified way. For instance, when applied to $\ell^1$-algebras of Fell bundles (the setting of Subsection \ref{weightt}), one recovers all the results in \cite{Fl24} that are stated for weighted algebras of discrete groups. In particular, we recover the results of \cite{FGLLM03,DLM04} that are stated for `polynomial weights'.

Of course, we not only recover old results but also produce new ones. Observe, for example, that all the preexisting literature, with the notable exception of \cite{SaSh19}, deals with groups of polynomial growth. However, we provide new results for weighted $\ell^p$-algebras of groups of strong subexponential growth (new even in the case $p=1$) and weighted $\ell^2$-algebras of groups with rapid decay (see Subsection \ref{weightttt}). In the latter case, groups are often nonamenable.

We now describe our results regarding the automatic continuity of intertwining operators. If $\X_1,\X_2$ are weak Banach $\A$-bimodules, then a linear map $\theta:\X_1\to\X_2$ is called an \emph{$\A$-intertwining operator} if for each $b\in\A$, the maps 
$$
\X_1\ni \xi\mapsto \theta(b\xi)-b\theta(\xi)\in \X_2\quad\text{ and }\quad \X_1\ni \xi\mapsto \theta(\xi b)-\theta(\xi)b\in \X_2
$$ 
are continuous. We recall that derivations, algebra homomorphisms, and bimodule homomorphisms are natural examples of intertwining operators.

The continuity of intertwining derivations is featured in works by Bade and Curtis \cite{BC91}, Runde \cite{Ru96,Ru98} and the author \cite{Fl25}. The continuity of particular classes of intertwining operators has been the topic of many other investigations, and we mention some of them now. For example, we have the results on automatic continuity in group algebras obtained by Willis \cite{Wi92}, Runde \cite{Ru00} and Dales \cite{Da00}. 

Another interesting phenomenon is that of automatic continuity on a dense subspace or subalgebra \cite{Wi01}. Runde even stated it as a question to find dense domains of continuity inside group algebras \cite[pag. 26]{Ru94}. Since many of our examples of $(k,p,q)$-differential subalgebras are, in fact, given by algebras of Beurling type, our results can be relevant to this question.

Based on the approach outlined in \cite{Fl25}, and the fact that $(k,p,q)$-differential subalgebras are symmetric, $^*$-regular and have the Wiener property, we are able to prove the following theorems.

\begin{thm}
    Let $\A$ be a $(k,p,q)$-differential subalgebra of the $C^*$-algebra $\B$, with the property that all closed $^*$-ideals of finite codimension have a bounded left approximate identity. Let $\X$ be a weak Banach $\A$-bimodule and $\theta:\A\to\X$ an $\A$-intertwining operator. Then $\theta$ is continuous if and only if $\mathscr I(\theta)$ is closed. In particular, $\theta$ is continuous if $\X$ is a Banach $\A$-bimodule.
\end{thm}

\begin{thm}
    Let $\A$ be a unital $(k,p,q)$-differential subalgebra of the $C^*$-algebra $\B$. Also consider a Banach $\A$-bimodule $\X_1$, a weak Banach $\A$-bimodule $\X_2$, and an $\A$-intertwining operator $\theta:\X_1\to\X_2$. If ${\rm C^*}(\A)$ has no proper closed two-sided ideals with finite codimension, then $\theta$ is continuous.
\end{thm}

Let us now describe a notable example. If one applies the previous theorem to, say, the setting of Subsection \ref{weightttt}, one finds that every intertwining operator over the natural weighted $l^2$-algebra of the free group $\mathbb F_n$ ($n\in\N\cup\{\infty\}$) is continuous. This fact remains true if one changes $\mathbb F_n$ by any nonamenable hyperbolic group, as their reduced $C^*$-algebras do not contain proper closed two-sided ideals with finite codimension. One should note that questions about automatic continuity for $\ell^1(\mathbb F_\infty)$ are central to the theory, as they relate to the (non-)existence of discontinuous homomorphisms \cite[Theorem 4.1]{Ru00}.

\section{The class of \texorpdfstring{$(k,p,q)$}--differential subalgebras}\label{exx}

\begin{defn}
    Let $(\A,\norm{\cdot}_{\A})$ be a Banach $^*$-subalgebra of the $C^*$-algebra  $(\B,\norm{\cdot}_{\B})$ such that the involution is continuous in $\A$. Let $k\in \Z^+\setminus\{1\}$ and $p,q>0$ such that $p+q=k$. We say that $\A$ is a {\it $(k,p,q)$-differential subalgebra} of $\B$ if there exists a constant $C>0$ such that
\begin{equation}\label{kpq}
    \|{a^{k}}\|_{\A}\leq C\norm{a}_{\A}^p\norm{a}_{\B}^q,\quad \textup{ for all } a\in\A.
\end{equation}
\end{defn}

The following lemma shows the result of applying the identity \eqref{kpq} successively.

\begin{lem}
    If $\A$ is a $(k,p,q)$-differential subalgebra of $\B$. Then there exists a constant $C>0$
    \begin{equation*}
        \|a^{k^{n}}\|_\A\leq C^n\norm{a}_\A^{p^n}\norm{a}_\B^{k^{n}-p^n}
    \end{equation*}
    holds for all $a\in\A$ and $n\in \Z^+$.
\end{lem}
\begin{proof}
    We will prove it by induction. Clearly $C$ is the constant given by the definition of $(k,p,q)$-differential subalgebra. For $n=1$, the identity is precisely the equality in \eqref{kpq}. For $n=m+1$, we have
    \begin{align*}
        \big\|a^{k^{m+1}}\big\|_{\A}&=\big\|(a^k)^{k^m}\big\|_{\A} \\
        &\leq C^m \big\|a^{k}\big\|_\A^{p^m}\big\|a^{k}\big\|_\B^{k^{m}-p^m} \\
        &\leq C^{m+1}\norm{a}_\A^{p^{m+1}}\norm{a}_\B^{p^mq}\norm{a}_\B^{k^{m+1}-kp^m} \\
        &= C^{m+1}\norm{a}_\A^{p^{m+1}}\norm{a}_\B^{k^{m+1}+(q-k)p^m}=C^{m+1}\norm{a}_\A^{p^{m+1}}\norm{a}_\B^{k^{m+1}-p^{m+1}},
    \end{align*} which proves the claim.
\end{proof}

Let us now argue that this class of Banach $^*$-algebras is vast and interesting. Obviously, we plan to do so by providing lots of examples of different natures, starting with what we call `normable ideals'.

\subsection{Normable ideals}\label{normable}

Let $\B$ be a $C^*$-algebra. A two-sided $^*$-ideal $I\subset \B$ is called a \emph{normable ideal} if there exists a norm $\norm{\cdot}_I$ on $I$ that makes it a Banach $^*$-algebra and such that $I\ni a\mapsto a^*\in I$ is continuous and
$$ 
\norm{ax}_I\leq \norm{a}_B\norm{x}_I,\quad\text{ and }\quad\norm{xa}_I\leq \norm{x}_I\norm{a}_\B,\quad \text{ for all }a\in\B,x\in I.
$$ 
If $I\subset \B$ is a normable ideal, then it is obviously a $(2,1,1)$-differential subalgebra. We now present two examples of normable ideals.

\begin{ex} \label{Schatten}
    Given a Hilbert space $\mathcal H$ and $p\in[1,\infty)$, the class of Schatten $p$-operators $S_p(\mathcal H)$ is a dense $^*$-ideal of the $C^*$-algebra of compact operators $\mathbb K(\mathcal H)$; moreover, it is a normable ideal under the standard Schatten $p$-norm $$\norm{T}_{S_p(\mathcal H)}={\rm Tr}(|T|^p)^{1/p}.$$
\end{ex}

\begin{ex} \label{hilbert}
    A proper $H^*$-algebra $\A$ is a Banach $^*$-algebra which is also a Hilbert space under the same norm and where $a^*$ is the unique element in $\A$ satisfying \begin{equation*}
        \langle ab,c\rangle=\langle b,a^*c\rangle\textup{ 
 and  }\langle ba,c\rangle=\langle b,ca^*\rangle\textup{, for all }b,c\in\A.
\end{equation*} 
Then the left regular representation $\A\ni a\mapsto L_a\in \mathbb B(\A)$ is a contractive  $^*$-monomorphism. We then have 
$$
\norm{ab}_{\A}\leq\norm{L_a}_{\mathbb B(\A)}\norm{b}_{\A}, \textup{ for all }a,b\in\A.
$$
Furthermore, the involution map $a\mapsto a^*$ must necessarily be isometric \cite[pg. 1211]{Pa01} and therefore the reversed inequality also holds: 
$$
\norm{ab}_{\A}\leq\norm{a}_{\A}\norm{L_b}_{\mathbb B(\A)}, \textup{ for all }a,b\in\A.
$$
So that $\A$ is a differential subalgebra of $\B=\overline{L(A)}^{\norm{\cdot}_{\mathbb B(\A)}}\cong {\rm C^*}(\A)$. In particular, when $K$ is a compact group, equipped with the normalized Haar measure, $L^2(K)$ becomes a proper $H^*$-algebra, with the operations inherited from $L^1(K)$.
\end{ex}

\subsection{Full Hilbert algebras}

Example \ref{hilbert} can be generalized to the class of Hilbert algebras, and we will do that here. The following definition is classical and taken from \cite[pg. 1214]{Pa01}. For a general treatment of Hilbert algebras and their properties, we recommend \cite[Chapter 11.7]{Pa01}.

\begin{defn}
    A \emph{Hilbert algebra} is a $^*$-algebra $\A$ that is a pre-Hilbert space under the same linear structure and an inner product $\langle\cdot,\cdot\rangle$ satisfying: \begin{itemize}
        \item[(i)] One has $\langle ab,c\rangle=\langle b,a^*c\rangle,$ for all $a,b,c\in\A$.
        \item[(ii)] For every $a\in\A$, there exists a constant $C_a>0$ such that $\norm{ab}_\H\leq C_a\norm{b}_\H$, for all $b\in\A$. Here $\norm{a}_\H=\langle a,a\rangle^{1/2}$.
        \item[(iii)] $\langle a,b\rangle=\langle b^*,a^*\rangle,$ for all $a,b\in\A$.
        \item[(iv)] $\A^2=\Span\{ab\mid a,b\in \A\}$ is dense in $\A$, under the topology given by the inner product.
    \end{itemize}
\end{defn}

Hilbert algebras are a central object of study in von Neumann algebra theory, at least in the setting of type $\mathrm{III}$ factors and Tomita-Takesaki theory (see \cite[Chapter VI]{Ta03}). The most obvious example is precisely $L^2(G)$, for a compact group $G$. The following example takes that idea as a starting point.

\begin{ex}
    let $G$ be an unimodular locally compact group and take $\A=C_{\rm c}(G)$, the set of compactly supported, complex-valued functions on $G$. Then $\A$ is a Hilbert algebra under the operations 
    $$
    \Phi*\Psi(x)=\int_G \Phi(y)\Psi(y^{-1}x)\d y\quad\text{ and }\quad \Phi^*(x)=\Phi(x^{-1})^*,
    $$ 
    valid for all $\Phi,\Psi\in \A$ and $x\in G$. The inner product is obviously given by 
    $$
    \langle\Phi,\Psi\rangle=\int_G \Phi(x)\overline{\Psi(x)}\d x, \quad\forall \Phi,\Psi\in\A.
    $$
\end{ex}

It is obvious from the definition that, on a Hilbert algebra $\A$ whose Hilbert space completion is $\H$, every element $a\in\A$ defines bounded operators $L_a,R_a\in\BofH$ given by the extension of left/right multiplication:
$$
L_a(b)=ab\quad\text{ and }\quad R_a(b)=ba,\quad\forall b\in\B.
$$
The definition of full Hilbert algebras, however, depends on the notion of bounded vectors, which is a slight generalization.

\begin{defn}
    Let $\A$ be a Hilbert algebra with Hilbert space completion $\mathcal H$. For $x\in\H$, define the operators $L_x,R_x:\A\subset \H\to\H$ by 
    $$
L_x(a)=R_a(x)\quad\text{ and }\quad R_x(a)=L_a(x),\quad\forall a\in\B.
$$
    The vector $x\in\H$ is said to be bounded if both $L_x$ and $R_x$ define bounded operators in $\BofH$. The set of bounded vectors in $\H$ is denoted by $\overline{\A}$, and we say that $\A$ is \emph{full} if $\A=\overline{\A}$.
\end{defn}

If $\A$ is a full Hilbert algebra, then endowing it with the norm 
$$
\norm{a}_\A=\norm{a}_\H+\norm{L_a}_{\BofH},
$$
makes $\A$ a Banach $^*$-algebra with isometric involution \cite[Theorem 11.7.10]{Pa01}. It is obvious that $\A$ is a $(2,1,1)$-differential subalgebra of $\B=\overline{\{L_a\mid a\in\A\}}^{\norm{\cdot}_{\BofH}}$. Indeed, one has 
\begin{align*}
    \norm{a^2}_\A&=\norm{a^2}_\H+\norm{L_{a^2}}_{\BofH} \\
    &\leq \norm{L_a}_{\BofH}\norm{a}_\H+\norm{L_a}_{\BofH}^2\leq \norm{L_a}_{\BofH}\norm{a}_\A,
\end{align*}
for all $a\in\A$.

\subsection{Domains of closed \texorpdfstring{$^*$}--derivations}\label{derivations}

Following \cite{BR75,KS93}, let us define (densely defined) unbounded derivations on $C^*$-algebras. In order to distinguish the situation from the (everywhere defined) derivations of Section \ref{autcont}, we will denote them with the symbol $\delta$, instead of just $D$.

Let $\B$ be a $C^*$-algebra. A closed derivation $\delta$ of $\B$ is a linear mapping from a dense subalgebra $\A$ into $\B$ that satisfies \begin{enumerate}
    \item[(i)] $\delta(ab)=a\delta(b)+\delta(a)b$, for all $a,b\in\A$.
    \item[(ii)] If $a_n\in\A$, $a_n\to a$ and $\delta(a_n)\to b$, then $a\in \A$ and $b=\delta(a)$.
\end{enumerate} If, in addition, $x\in \A$ implies $x^*\in\A$ and $\delta(x^*)=\delta(x)^*$, then $\delta$ is called a closed $^*$-derivation. 

If $\delta$ is a closed $^*$-derivation, then $\A$ is a Banach $^*$-algebra when endowed with the norm $$\norm{a}_\A=\norm{a}_\B+\norm{\delta(a)}_\B.$$

The obvious example of a closed $^*$-derivation is given by $\delta(f)=f'$ on $\B=C(\T)$, and in this case, $\A$ is precisely $C^1(\T)$, with its usual norm. For this reason, derivations are usually viewed as a noncommutative differential operator and therefore form a central object in the study of noncommutative geometry \cite{NCG}.

It is easy to observe that $\A$ is a $(2,1,1)$-differential subalgebra of $\B$. Indeed, it is shown by the following simple calculation, valid for all $a\in\A$. 
\begin{align*}
    \norm{a^2}_\A&=\norm{a^2}_\B+\norm{\delta(a^2)}_\B \\
    &\leq \norm{a}_\B^2+\norm{a\delta(a)}_\B+\norm{\delta(a)a}_\B \\
    &\leq \norm{a}_\B\big(\norm{a}_\B+2\norm{\delta(a)}_\B\big)\leq  2\norm{a}_\B\norm{a}_\A.
\end{align*}

\subsection{Weighted \texorpdfstring{$\ell^1$}--algebras of Fell bundles over discrete groups with polynomial growth}\label{weightt}
In this section, $ G$ will be a discrete group with unit $\e$. We recall that $ G$ has \emph{polynomial growth} if, for every finite subset $K\subset G$, there exists a $d\in\N$ such that
$$
\mu(K^n)=O(n^d),\quad\textup{ as }n\to\infty.
$$ 

Let us fix a Fell bundle $\mathscr C\!=\bigsqcup_{x\in G}\mathfrak C_x$ over $G$. The associated algebra of summable cross-sections $\ell^1( G\,\vert\,\mathscr C)$ is a Banach $^*$-algebra and a completion of the space $C_{\rm c}( G\,\vert\,\mathscr C)$ of cross-sections with finite support under the norm 
$$
\norm{\Phi}_1=\sum_{x\in G}\norm{\Phi(x)}_{\mathfrak C_x}.
$$ 
The universal $C^*$-algebra completion of $\ell^1( G\,\vert\,\mathscr C)$ is denoted by ${\rm C^*}(G\,\vert\,\mathscr C)$. For the general theory of Fell bundles, we recommend \cite[Chapter VIII]{FD88}. Here, we only recall the product on $\ell^1( G\,\vert\,\mathscr C)$, given by
\begin{equation*}\label{broduct}
\big(\Phi*\Psi\big)(x)=\sum_{y\in G} \Phi(y)\bu \Psi(y^{-1}x)
\end{equation*}
and its involution
\begin{equation*}\label{inwol}
\Phi^*(x)=\Phi(x^{-1})^\bu\,,
\end{equation*}
in terms of the operations $\big(\bu,^\bu\big)$ on the Fell bundle. Note that the adjoint operation $\Phi\mapsto \Phi^*$ is isometric.

\begin{defn}
    A {\rm weight} on the group $G$ is a function $\nu: G\to [1,\infty)$ satisfying 
\begin{equation*}\label{submultiplicative}
 \nu(xy)\leq \nu(x)\nu(y)\,,\quad\nu(x^{-1})=\nu(x)\,,\quad\forall\,x,y\in G.
\end{equation*} 
In addition, the weight $\nu$ is said to be a \emph{polynomial weight} if there is a constant $C>0$ such that \begin{equation*}
        \nu(xy)\leq C\big(\nu(x)+\nu(y)\big),
    \end{equation*} for all $x,y\in G$.
\end{defn}

If $G$ is of polynomial growth and finitely generated or discrete and locally finite, then it is possible to construct a polynomial weight $\nu$ on $ G$ such that $\nu^{-1}\in \ell^p( G)$, for any $0<p<\infty$ \cite{Py82}.

Given a weight $\nu$ on $G$, one may consider the weighted subalgebra $\ell^{1,\nu}( G\,\vert\,\mathscr C)$, which consists of all $\Phi\in\ell^{1}( G\,\vert\,\mathscr C) $ such that
$$
\norm{\Phi}_{1,\nu}:=\sum_{x\in G}\nu(x)\norm{\Phi(x)}_{\mathfrak C_x}<\infty.
$$
It is a Banach $^*$-subalgebra of $\ell^{1}( G\,\vert\,\mathscr C)$ under the inherited operations and the norm $\norm{\cdot}_{1,\nu}$. Furthermore, in \cite[Proposition 5.15, Lemma 3.4]{Fl24} we proved the following.

\begin{prop}[\cite{Fl24}]\label{Fell}
    Let $\nu$ a polynomial weight on $G$ such that $\nu^{-1}$ belongs to $\ell^p( G)$, for some $0<p<\infty$. Then $\ell^{1,\nu}( G\,\vert\,\mathscr C)$ is a $(4,\tfrac{4p+3}{p+1},\tfrac{1}{p+1})$-differential subalgebra of ${\rm C^*}( G\,\vert\,\mathscr C)$.
\end{prop}

\subsection{Weighted subalgebras of groups with strong subexponential growth}\label{weighttt}

Let $G$ be a finitely generated group and let $S$ denote a finite, symmetric set of generators. Suppose that $G$ has \emph{strong subexponential growth}, that is, suppose there exists an $\alpha\in(0,1)$ such that 
$$ 
|S^n|=O( e^{n^\alpha}),\quad \text{ as } n\to\infty.
$$
Using the set $S$, one also defines the word-length, which is the function $\tau_S:G\to[0,\infty)$, given by the formula 
\begin{equation*}\label{word-lenght}
    \tau_S(x)=\min\{n\in\N\mid x\in S^n\}.
\end{equation*}
As we did with $\ell^1$ in the previous subsection, one can form weighted versions of the $\ell^p$-spaces. Indeed, the space 
$$
\ell^{p,\nu}(G)=\{\Phi:G\to \CC\mid \sum_{x\in G}\nu(x)^p|\Phi(x)|^p<\infty\}
$$
is a Banach space under the norm 
$$
\norm{f}_{p,\nu}=\Big(\sum_{x\in G}\nu(x)^p|\Phi(x)|^p\Big)^{1/p}.
$$
Given a group of strong subexponential growth, Samei and Shepelska considered the weight given by \begin{equation}\label{SSh}
    \nu_{D,\alpha_0}(x)=De^{\tau_S(x)^{\alpha_0}},
\end{equation} 
where $D>0$ and $\alpha<\alpha_0<1$. Under these conditions, they studied the associated weighted $\ell^p$-space and showed not only that it is a Banach $^*$-algebra under convolution, but that it is a differential subalgebra of ${\rm C^*}(G)$ \cite[Theorem 4.1, Corollary 4.3]{SaSh19}.\footnote{Note that Samei and Shepelska use the notation $\ell^p(G,\nu)$ instead of $\ell^{p,\nu}(G)$. We chose the latter here to stay consistent with the notation in subsections \ref{weightt} and \ref{weightttt}.}

\begin{prop}[Samei-Shepelska \cite{SaSh19}]
    Let $p\in[1,\infty)$, $G$ be a finitely generated group of strong subexponential growth and $\nu=\nu_{D,\alpha_0}$ a weight of the form \eqref{SSh}. Then there exists $\theta\in(0,1)$ such that $\ell^{p,\nu}(G)$ is a $(2,1+\theta,1-\theta)$-differential subalgebra of ${\rm C^*}(G)$.
\end{prop}

\subsection{Sobolev algebras of groups with the rapid decay property}\label{weightttt}
    Let $G$ be a group. A length function $l:G\to[0,\infty)$ such that $l(\e)=0$, $l(x^{-1})=l(x)$ and $l(xy)\leq l(x)+l(y)$, for all $x,y\in G$. Given a length function, we define the ball centered at $y\in G$ and of radius $r\geq 0$ by 
    $$
    B(y,r)=\{x\in G\mid l(xy^{-1})\leq r\}.
    $$
    The left regular representation is defined by convolution with $\ell^2(G)$ functions, that is, for $\Phi\in C_{\rm c}(G)$,
    $$
    \lambda(\Phi)\Psi=\Phi*\Psi,\quad \text{ for all }\Psi\in \ell^2(G).
    $$
    The $C^*$-algebra generated by $\{\lambda(\Phi)\mid \Phi\in C_{\rm c}(G)\}$ is called the \emph{reduced group $C^*$-algebra} and we denote it by ${\rm C_r^*}(G)$. 
    \begin{defn}
        We say that the group $G$ has \emph{rapid decay} or that it has \emph{(RD)} with respect to the length function $l:G\to[0,\infty)$ if there exists a polynomial $P$ such that for every $r\geq 0$ and for every $\Phi\in C_{\rm c}(G)$ with ${\rm Supp}(\Phi)\subset B(\e,r)$, one has
        $$
        \norm{\Phi}_{{\rm C_r^*}(G)}\leq P(r)\norm{\Phi}_{\ell^2(G)}.
        $$
    \end{defn}
    The rapid decay property is enjoyed by various classes of groups (see \cite{Ch17}), notably finitely generated free groups and hyperbolic groups. In \cite{La00}, Lafforgue proved that groups with the rapid decay property have a distinguished algebra with notable K-theoretic properties (see \cite{La02}). The algebra in question is constructed as follows. For large enough $s>0$, consider the weight 
    $$
    \nu_s(x)=(1+l(x))^s.
    $$
    And the algebra in question is, in the notation of the previous section, $\ell^{2,\nu_s}(G)$.\footnote{Lafforgue uses the notation $H^s(G)$, since this algebra is historically inspired by the space of Sobolev functions. However, we wanted to keep consistency with the notation of the previous subsections.} The fact that this $\ell^2$-space is an algebra is not obvious.

    \begin{thm}[\textcite{La00}]
        Let $G$ be a group with the rapid decay property with respect to the length function $l:G\to[0,\infty)$. Then there exists a constant $s>0$ such that $\ell^{2,\nu_s}(G)$ is a $(2,1,1)$-differential subalgebra of ${\rm C^*_r}(G)$.
    \end{thm}

\subsection{Matrix algebras with polynomial off-diagonal decay}
    The following are examples of algebras of matrices which are generalized differential subalgebras of $\mathbb B(\ell^2(\Z))$. Our main reference for this subsection is \cite{ShSu20}. See also \cite{Su05,GL06,Su07}.

Given $p\in [1,\infty]$ and $\alpha\geq 0$, we define the Gr\"ochenig-Schur family of infinite matrices by
 \begin{equation*}
{\mathcal A}_{p,\alpha}=\Big\{ A=(a(i,j))_{i,j \in \Z} \mid  \|A\|_{{\mathcal A}_{p,\alpha}}<\infty\Big\},
\end{equation*}
the Baskakov-Gohberg-Sj\"ostrand family of infinite matrices by
\begin{equation*}
{\mathcal C}_{p,\alpha}=\Big\{ A= (a(i,j))_{i,j \in \Z} \mid  \|A\|_{{\mathcal C}_{p,\alpha}}<\infty\Big\},
\end{equation*}
and the Beurling family of infinite matrices
\begin{equation*}
{\mathcal B}_{p,\alpha}=\Big\{ A= (a(i,j))_{i,j \in \Z} \mid  \|B\|_{{\mathcal A}_{p,\alpha}}<\infty\Big\}.
\end{equation*}
The norms involved above are given by
 \begin{equation*}
 \|A\|_{{\mathcal A}_{p,\alpha}}
 = \max \Big\{ \sup_{i \in \Z} \big\|\big(a(i,j) \nu_\alpha(i, j)\big)_{j\in \Z}\big\|_{\ell^p(\Z)}, \ \ \sup _{j \in \Z}
 \big\|\big(a(i,j) \nu_\alpha(i, j)\big)_{i\in \Z}\big\|_{\ell^p(\Z)}
 \Big\},
\end{equation*}
\begin{equation*}
\|A\|_{{\mathcal C}_{p,\alpha}} = \Big\| \Big(\sup_{i-j=k} |a(i,j)| \nu_\alpha(i, j)\Big)_{k\in \Z} \Big\|_{\ell^p(\Z)},
\end{equation*}
and
\begin{equation*}
\|A\|_{{\mathcal B}_{p,\alpha}} = \Big\| \Big(\sup_{|i-j|\ge |k| } |a(i,j)| \nu_\alpha(i, j)\Big)_{k\in \Z} \Big\|_{\ell^p(\Z)},
\end{equation*}
where $\nu_\alpha(i, j)=(1+|i-j|)^\alpha$, is a polynomial weight on $\Z^2$. It is clear that
\begin{equation*}
{\mathcal B}_{p,\alpha} \subset {\mathcal C}_{p,\alpha} \subset {\mathcal A}_{p,\alpha}, \quad\text{ for all }1\leq p\leq \infty  \text{ and all }\alpha\geq 0,
\end{equation*}
where the inclusions are proper as long as $p\not=\infty$. In the case $p=\infty,$ one actually has 
$$
{\mathcal B}_{\infty,\alpha}={\mathcal C}_{\infty,\alpha}={\mathcal A}_{\infty,\alpha},
$$
and the resulting algebra is called the Jaffard algebra, and it is usually denoted by ${\mathcal J}_{\alpha}$. The following proposition is taken from \cite[Theorem 2.5]{ShSu20}.

\begin{prop}[Shin-Sun \cite{ShSu20}]
Let $1\leq p\leq\infty$ and $\alpha>1-1/p$. Denote $$f(p,\alpha)=\frac{\alpha+1/p-1}{\alpha+1/p-1/2}\in (0, 1).$$ Then  ${\mathcal A}_{p,\alpha}$, ${\mathcal C}_{p,\alpha}$ and ${\mathcal B}_{p, \alpha}$ are $(2,2-f(p,{\alpha}),f(p,{\alpha}))$-differential subalgebras of ${\mathbb B}(\ell^2(\Z))$.
\end{prop}

\subsection{Algebras of localized integral operators} The same paper of Shin and Sun \cite{ShSu20} also provides us with an example of an algebra satisfying \eqref{4diff}. This algebra is of a different nature than the previous ones, and so it deserves to be discussed by itself. The algebra in question was first introduced in \cite{Su08}.

Fix $1\leq p\leq \infty, \alpha\geq 0$ and $\gamma\in [0, 1)$, we define the norm
of a kernel  $K:\R\times\R\to\C$ by
\begin{equation*}
\|K\|_{{\mathcal W}_{p,\alpha}^\gamma}=
\left\{
\begin{array}{ll}
 \max\Big(\sup_{x\in \R} \big\|K(x,\cdot)\nu_\alpha(x,\cdot)\big\|_{L^p(\R)},\
\sup_{y\in \R} \big\|K(\cdot,y)\nu_\alpha(\cdot,y)\big\|_{L^p(\R)}\Big)  & {\rm  if }\ \gamma =0
\\
\|K\|_{{\mathcal W}_{p,\alpha}^0}+\sup_{0<\delta\le 1} \delta^{-\gamma}
\|\omega_\delta(K)\|_{{\mathcal W}_{p,\alpha}^0} & {\rm if } \ 0 < \gamma <1,
\end{array}
\right.
\end{equation*}
where  the modulus of continuity of the kernel $K$ is defined by
\begin{equation*}
\omega_\delta(K)(x,y):=\sup_{|x^\prime|\leq \delta, |y^\prime|\le \delta}|K(x+x^\prime, y+y^\prime)-K(x,y)|,\quad \text{ for all } x, y\in \R,
\end{equation*}
and $\nu_\alpha(x, y)= (1+|x-y|)^\alpha,  x, y\in \R$. Consider the set ${\mathcal W}_{p,\alpha}^\gamma$ of integral operators
\begin{equation*}
Tf(x)=\int_{{\R}} K_T(x,y) f(y) dy, \quad f \in L^2(\R),
\end{equation*}
whose kernels $K_T$ satisfy $\|K_T\|_{{\mathcal W}^\gamma_{p, \alpha}}<\infty$, and define
$$
\|T\|_{{\mathcal W}_{p,\alpha}^\gamma}:=
\|K_T \|_{{\mathcal W}_{p,\alpha}^\gamma}, \ T\in {\mathcal W}_{p,\alpha}^\gamma.
$$
Integral operators in ${\mathcal W}_{p, \alpha}^\gamma$ have their kernels being H\"older continuous of order $\gamma$ and having off-diagonal polynomial decay of order $\alpha$. For $1\leq p\leq \infty$, $0\leq \gamma<1$ and $\alpha>1-1/p$, one may verify that ${\mathcal W}_{p, \alpha}^\gamma$ are Banach $^*$-subalgebras of ${\mathcal B}(L^2)$ under operator composition. Furthermore, thanks to \cite[Theorem 3.3]{ShSu20}, we have the following.

\begin{prop}[Shin-Sun \cite{ShSu20}]
    Let $1\leq p\leq \infty$, $0\leq \gamma<1$, and $\alpha>1-1/p$. Denote $$f(p,\gamma,\alpha)=\frac{\alpha+\gamma+1/p}{(1+\gamma)(\alpha+1/p).}$$ Then ${\mathcal W}_{p, \alpha}^\gamma$ is a $(4,4-f(p,\gamma,\alpha),f(p,\gamma,\alpha))$-differential subalgebra of $\mathbb B(L^2(\R))$.
\end{prop}

In particular, the algebra ${\mathcal W}_{p, \alpha}^\gamma$ is an example of \eqref{4diff}.

\section{Aspects of harmonic analysis}

Now that we have established the relevant definition and shown a plenitude of examples, it is time to develop the systematic study that we promised in the introduction.

\subsection{Symmetry}\label{symm} 

In what follows, $\widetilde \A$ will denote the minimal unitization of $\A$ (if $\A$ was already unital, we simply set $\widetilde \A=\A$). Symmetry is a spectral property, so it seems natural to fix our notations for the spectrum and the spectral radius now. As usual, 
$$
{\rm Spec}_{\A}(a)=\{\lambda\in\CC\mid a-\lambda1\textup{ is not invertible in }\widetilde\A\}
$$
will denote the spectrum of an element $a\in\A$, while 
$$
\rho_\A(a)=\sup\{|\lambda|\mid \lambda\in {\rm Spec}_{\A}(a) \}
$$ 
denotes its spectral radius. Gelfand's formula for the spectral radius also gives the following, very useful, equality
$$
\rho_\A(a)=\lim_{n\to\infty} \norm{a^{n}}_{\A}^{1/n}.
$$ 
\begin{defn} \label{symmetric}
A Banach $^*$-algebra $\mathfrak A$ is called {\it symmetric} if the spectrum of $a^*a$ is non-negative for every $a\in\mathfrak A$.
\end{defn}

In fact, the celebrated Shilari-Ford theorem says that a Banach $^*$-algebra $\mathfrak A$ is symmetric if and only if the spectrum of any self-adjoint element is real \cite[Theorem 11.4.1]{Pa01}. 

The idea now is to prove that $(k,p,q)$-differential subalgebras are symmetric. For that purpose, we recall the famous lemma of Barnes-Hulanicki \cite{Ba90}.

\begin{lem}[Barnes-Hulanicki]\label{BH}
Let $\A$ be a Banach $^*$-algebra, $S$ a $^*$-subalgebra of $\A$ and $\pi : \A\to \BofH$ a faithful $^*$-representation. If $\A$ is unital, we assume that the inclusion $\pi(\A)\subset \BofH$ is unital.
If
$$
\rho_\A(a)=\|\pi(a)\|_{\BofH},
$$
for all $a\in S_{\rm sa}$, then
$$
{\rm Spec}_\A(a)={\rm Spec}_{\BofH}(\pi(a))
$$
for every $a\in S$.
\end{lem}

\begin{thm}\label{symmetry}
    Let $(k,p,q)$-differential subalgebra of $\B$. Then
    $$
{\rm Spec}_\A(a)={\rm Spec}_{\B}(a)
$$
for all $a\in\A$. In particular, $\A$ is symmetric.
\end{thm}
\begin{proof}
    By the GNS construction, we can find a (necessarily isometric) faithful $^*$-representation $\pi : \B\to \BofH$ and so we identify $\B$ with $\pi(\B)$. Then we use the definition of differential subalgebra to compute the spectral radius using Gelfand's formula. Indeed, for $a\in\A$ and $n\in \Z^+$, one has
    $$
    \|a^{kn}\|_\A\leq C\|a^n\|_\A^{p}\|a^n\|_\B^{{k}-p},
    $$
    so 
    $$
    \rho_\A(a)=\lim_{n\to\infty} \|a^{kn}\|_{\A}^{1/(kn)}\leq \lim_{n\to\infty}C^{1/(kn)}\|a^n\|_\A^{p/(kn)}\|a^n\|_\B^{q/(kn)}=\rho_\A(a)^{p/k}\rho_\B(a)^{q/k},
    $$
    implying that $\rho_\A(a)^{1-p/k}=\rho_\A(a)^{q/k}\leq \rho_\B(a)^{q/k}$ and $\rho_\A(a)\leq \rho_\B(a)$. The reverse inequality trivially holds, so we get $\rho_\A(a)=\rho_\B(a)$. For a self-adjoint element $a$, one has $\rho_\B(a)=\|a\|_\B$, so the proof finishes by invoking the Barnes-Hulanicki Lemma. 
\end{proof}

The result we just obtained says that the spectrum of an element $a\in\A\subset \B$ is independent of the algebra of reference. This implies that the holomorphic functional calculi of both algebras coincide. The idea behind the following section is to prove the same about a slightly bigger functional calculus.

\subsection{Stability under the functional calculus of \texorpdfstring{$\B$}-}\label{funcalc}

As in \cite{Fl25,Fl24}, the growth of $1$-parameter unitary groups plays a major role in the development of our results. In order to include the non-unital case, we are forced to consider the entire function $u:\mathbb C\to \mathbb C$, given by 
\begin{equation*}\label{functions}
    u(z)=e^{iz}-1=\sum_{k=1}^{\infty} \frac{i^k z^k}{k!}
\end{equation*} 
and replace $e^{ia}$ by $u(a)$ to avoid unnecessary unitizations. 

The next lemma was inspired by \cite[Lemma 3]{Py82} and \cite[Lemma 4.7]{Fl25}, which are related to the cases $k=2$ and $k=4$, respectively. In fact, those lemmas were made for slightly worse-behaved algebras and are slightly less pleasant to prove (compare condition \emph{(ii)}).

\begin{lem}\label{asymp}
    Let $k\in\Z^+$, $1<\gamma<k$ and let $\{a_n\}_{n=1}^\infty$ be a sequence of non-negative real numbers such that \begin{enumerate}
        \item[(i)] $a_{n+m}\leq a_na_m$ and
        \item[(ii)] $a_{kn}\leq a_n^\gamma,$
    \end{enumerate} for all $n$. Then for all $\tau>\log_{k}(\gamma)$, one has $a_n=O(e^{n^\tau}),\textup{ as }n\to\infty$.
\end{lem}

\begin{proof}
Because of \emph{(ii)}, we have, for $i\in\N$,
$$
a_{k^i}\leq a_{1}^{\gamma^i}.
$$ 
For a general $n\in\N$, we consider its $k$-adic expansion $n=\sum_{i=0}^m \epsilon_i k^i$, where $\epsilon_i\in\{0,\ldots, k-1\}$ and $\log_{k}(n)\leq  m<1+\log_{k}(n)$. Denoting $A=\ln(a_1)$, we see that 
\begin{equation}\label{boundd}
    a_n\leq \prod_{i=0}^m a_{k^i}^{\epsilon_i}\leq \prod_{i=0}^m a_1^{\epsilon_i\gamma^i}\leq a_1^{({k}-1)(m+1)\gamma^m}\leq e^{A(k-1)(2+\log_{k}(n))kn^{\log_{k}(\gamma)}}, 
\end{equation}
from which the desired property follows.
\end{proof}

\begin{prop}\label{growth}
    Let $\A$ be a $(k,p,q)$-differential subalgebra of $\B$. Then for every $a=a^*\in \A$, one has 
    $$
    \norm{u(ta)}_{\A}=O(e^{|t|^\tau}),\textup{ as }|t|\to\infty,
    $$ 
    for all $\tau>\log_{k}(\max\{k-1,p\})$.
\end{prop}
\begin{proof}
    We will, of course, make use of Lemma \ref{asymp}. For that matter, let $\gamma>1$ be such that $\max\{k-1,p\}\leq\gamma\leq k$. Consider the sequence $a_n=D(\norm{u(na)}_{\A}+1)$, with $D\geq 1$ to be determined later. One then has 
    \begin{align*}
        a_{n+m}&=D(\norm{u(na)u(ma)+u(na)+u(ma)}_{\A}+1) \\
        &\leq D(\norm{u(na)}_{\A}+1)(\norm{u(ma)}_{\A}+1) \leq a_na_m.
    \end{align*} 
    Proving part \emph{(ii)} is a lot more involved. We need to make some heavy computations, so let us introduce the following auxiliary notations 
    $$
    z:=u(na)\quad\text{ and }\quad  x:=e^{ina}=z+1.
    $$
    Also note that
    \begin{equation}\label{combinatorics}
        \sum_{\beta=0}^{k-2}\binom{k-1}{\beta}+ \sum_{\alpha=0}^{k-2}\sum_{\beta=0}^\alpha\binom{\alpha}{\beta}=2^{k}-2.
    \end{equation}
    Using the above-mentioned notation, we see that
    \begin{align*}
        u(kna)&=(e^{ina})^{k}-1=x^{k}-1=(x-1)\sum_{\alpha=0}^{k-1}x^\alpha  \\
        &=z\sum_{\alpha=0}^{k-1}(z+1)^\alpha=\sum_{\alpha=0}^{k-1}\sum_{\beta=0}^{\alpha} \binom{\alpha}{\beta}z^{\beta+1}  \\
        &=\sum_{\beta=0}^{k-1} \binom{k-1}{\beta}z^{\beta+1}+\sum_{\alpha=0}^{k-2}\sum_{\beta=0}^{\alpha} \binom{\alpha}{\beta}z^{\beta+1} \\
        &=z^{k}+\sum_{\beta=0}^{k-2} \binom{k-1}{\beta}z^{\beta+1}+\sum_{\alpha=0}^{k-2}\sum_{\beta=0}^{\alpha} \binom{\alpha}{\beta}z^{\beta+1}
    \end{align*}
    Now we consider separately the cases $\norm{u(na)}_{\A}\leq 1$ and $\norm{u(na)}_{\A}>1$. In the first case, we have
    \begin{align*}
        a_{kn}&=D(\|u(kna)\|_{\A}+1) \\
        &= D\Big(\Big\|z^{k}+\sum_{\beta=0}^{k-2} \binom{k-1}{\beta}z^{\beta+1}+\sum_{\alpha=0}^{k-2}\sum_{\beta=0}^{\alpha} \binom{\alpha}{\beta}z^{\beta+1}\Big\|_{\A}+1\Big)  \\
        &\leq D\Big(2+\sum_{\beta=0}^{k-2} \binom{k-1}{\beta}+\sum_{\alpha=0}^{k-2}\sum_{\beta=0}^{\alpha} \binom{\alpha}{\beta}\Big)\\
        &\overset{\eqref{combinatorics}}{=} 2^{k}D\\
        &\leq 2^{k}D(\norm{u(na)}_{\A}+1)^{\gamma}.
    \end{align*}
    Now, if $\norm{u(na)}_{\A}> 1$, then we have
    \begin{align*}
        a_{kn}&=D\Big(\Big\|z^{k}+\sum_{\beta=0}^{k-2} \binom{k-1}{\beta}z^{\beta+1}+\sum_{\alpha=0}^{k-2}\sum_{\beta=0}^{\alpha} \binom{\alpha}{\beta}z^{\beta+1}\Big\|_{\A}+1\Big) \\
        &\leq D\Big(\|{u(na)^{k}}\|_\A+\sum_{\beta=0}^{k-2} \binom{k-1}{\beta}\|{u(na)}\|_\A^{\beta+1}+\sum_{\alpha=0}^{k-2}\sum_{\beta=0}^{\alpha} \binom{\alpha}{\beta}\norm{u(na)}_\A^{\beta+1}+1\Big)  \\
        &\overset{\eqref{combinatorics}}{\leq } D\Big(\|{u(na)^{k}}\|_\A+(2^{k}-2)\norm{u(na)}_\A^{k-1}+ 1\Big)  \\
        &\overset{\eqref{kpq}}{\leq } D(2^qC\norm{u(na)}_{\A}^p+(2^{k}-1)\norm{u(na)}_\A^{k-1}) \\
        &\leq D(2^qC+2^{k}-1)\norm{u(na)}_\A^\gamma.
    \end{align*} 
    Therefore, setting $D^{\gamma-1}\geq \max\{2^{k}, 2^qC+2^{k}-1\}$, one concludes that $a_{kn}\leq a_n^\gamma$, irrespective of the value of $\norm{u(na)}_{\A}$. By invoking Lemma \ref{asymp}, we obtain the conclusion
    \begin{equation}\label{expan}
        \norm{u(na)}_{\A}=O(e^{n^\tau}),\textup{ as }n\to\infty.
    \end{equation}
    Now that we have the desired result for every $n\in\N$, it remains to argue that the general result follows from that. Indeed, for $t>0$ (the case $t<0$ is analogous), we take $n=\floor{t}$ and observe that 
    \begin{align}
        \norm{u(ta)}_\A&=\norm{u\big((t-n)a+na\big)}_\A \nonumber \\
        &=\norm{u((t-n)a)u(na)+u((t-n)a)+u(na)}_{\A} \nonumber \\
        &\leq \norm{u((t-n)a)}_\A\norm{u(na)}_\A+\norm{u(na)}_\A+\norm{u((t-n)a)}_\A \nonumber \\
        &\leq (e^{\norm{a}_\A}+2)\norm{u(na)}_\A+e^{\norm{a}_\A}+1, \label{expan1}
    \end{align}
    so $\norm{u(ta)}_\A$ is $O(\norm{u(\floor{t}a)}_\A)=O(e^{{\floor{t}}^\tau})=O(e^{t^\tau})$, as $t\to\infty$. The proof is finished.    \end{proof}

\begin{rem}\label{joingrowth}
    The above proposition shows that, for a fixed $a\in \A_{\rm sa}$ and a large enough $|t|$, there is a constant $B>0$ so that
    \begin{equation}\label{multiple}
        \norm{u(ta)}_{\A}\leq Be^{|t|^{\tau}}
    \end{equation}
    We remark that the proofs of Proposition \ref{growth} and Lemma \ref{asymp} reveal that the constant $B$ only depends on the quantity $\norm{a}_{\A}$. Indeed, the asymptotics \eqref{expan} are obtained from the usage of Lemma \ref{asymp}. However, if one looks at the proof of Lemma \ref{asymp}, the only constant that is not fixed by the data $k,p,q,C$ is the $A$ in \eqref{boundd}. But in the context of Proposition \ref{growth}, one has 
    $$
    A=\ln(\|u(a)\|_\A+1)+\ln(D)\leq \ln(e^{\|a\|_\A}+1)+\ln(D).
    $$
    The passage to a general $t\in \R$ done in \eqref{expan1} only depends on $\|a\|_\A$ and $\norm{u(n a)}_\A$.
    
    It therefore follows, from the same reasoning, that the growth conclusion \eqref{multiple} can be obtained simultaneously for a family of elements $a\in F\subset\A_{\rm sa}$, provided that $\sup_{a\in F} \norm{a}_{\A}<\infty$.
\end{rem}

Before introducing our theorem about functional calculus, let us introduce some notation. Let $\A$ be a Banach algebra. $\A(a_1,\ldots, a_n)$ denotes the closed subalgebra of $\A$ generated by the elements $a_1,\ldots, a_n\in\A$. If $\A$ is a commutative Banach algebra with spectrum $\Delta$, then $\hat a\in C_0(\Delta)$ denotes the Gelfand transform of $a\in \A$. In particular, the Fourier transform of a complex function $f:\mathbb R\to\CC$ is 
$$
\widehat{f}(t)=\int_{\mathbb R}f(x)e^{-itx}\d x.
$$

\begin{defn}
    Let $\A$ be a semisimple commutative Banach algebra with spectrum $\Delta$. $\A$ is called {\it regular} if for every closed set $X\subset \Delta$ and every point $\omega\in \Delta\setminus X$, there exists an element $a\in \A$ such that $\hat a(\varphi)=0$ for all $\varphi\in X$ and $\hat a(\omega)\not=0$.
\end{defn}

Regularity means, quite literally, that the Banach algebra $\A$, when identified as a subalgebra of $C_0(\Delta)$, contains enough functions to separate points from closed subsets. At an intuitive level, this means that a functional calculus based on functions in $\A$ will be very rich. It is proven in \cite[Proposition 4.1.18]{Da00} that regular algebras are also `normal'. This is a very useful feature, so we record it in the next remark.

\begin{rem}\label{normal}
    Let $\A$ be a regular commutative Banach algebra with spectrum $\Delta$. Then, for every closed set $X\subset \Delta$ and every compact subset $Y\subset\Delta\setminus X$, there exists an element $a\in \A$ such that $\hat a\equiv 0$ in $X$ and $\hat a\equiv1$ in $Y$.
\end{rem}

\begin{ass}
    From now on and until the end of the paper, given the data $(k,p,q)$, $\tau$ will always denote a positive number such that $1>\tau>\log_{k}(\max\{k-1,p\})$. Furthermore, $\mathfrak D_\tau$ will denote the algebra of functions 
$$
f(x)=\frac{1}{2\pi}\int_{\mathbb R} \widehat{f}(t)e^{itx}\d t,
$$ 
where $ \widehat{f}\in L^1(\R)$ and 
$$
    \norm{f}_{\tau}=\int_{\R}|\widehat{f}(t)|e^{|t|^{\tau}}\d t<\infty.
$$ 
Under the norm $\norm{\cdot}_{\tau}$, $\mathfrak D_\tau$ is a Banach subalgebra of $C_0(\R)$. As mentioned in the introduction, Domar proved that $\mathfrak D_\tau$ is a regular commutative Banach algebra \cite{Do56}.
\end{ass}

The following theorem is based on the smooth functional calculus construction attributed to Dixmier \cite[Lemme 7]{Di60} and Baillet \cite[Théorème 1]{MB79}. See also \cite[Theorem 2.5]{Fl24}.

\begin{thm}\label{DixBai}
    Let $\A$ be a $(k,p,q)$-differential subalgebra of $\B$. Let $1>\tau>\log_{k}(\max\{k-1,p\})$, $a=a^*\in \A$, and $f\in\mathfrak D_\tau$. Then the following are true.
    \begin{enumerate}
        \item[(i)] The following Bochner integral exists in $\widetilde{\A}(a,1)$: \begin{equation*}
            f(a)=\frac{1}{2\pi}\int_{\mathbb R} \widehat{f}(t)e^{ita}\d t.
        \end{equation*} Moreover, if $f(0)=0$, then $f(a)\in \A(a)$.
        \item[(ii)] For any non-degenerate $^*$-representation $\Pi:\widetilde{\A}(a,1)\to\BofH$, we have ${\Pi}\big(f(a)\big)=f\big(\Pi(a)\big)$. 
        \item[(iii)] ${\rm Spec}_{\widetilde{\A}(a,1)}\big(f(a)\big)=f\big({\rm Spec}_{\widetilde{\A}(a,1)}(a)\big)$.
        \item[(iv)] There is a (continuous) $^*$-homomorphism $\varphi_a:\mathfrak D_\tau\to \widetilde{\A}(a,1)$ defined by $\varphi_a(f)=f(a)$, such that $\varphi_a(g)=1$ if $g\equiv 1$ on a neighborhood of ${\rm Spec}_{\widetilde{\A}(a,1)}(a)$ and $\varphi_a(g)=a$ if $g(x)=x$ on a neighborhood of ${\rm Spec}_{\widetilde{\A}(a,1)}(a)$.
    \end{enumerate}
\end{thm}

\begin{proof} \emph{(i)} Because of Proposition \ref{growth}, there is a constant $A>0$ such that 
$$
\|e^{ita}\|_{\widetilde\A}\leq A(1+e^{|t|^{\tau}}),\quad \textup{ for all }t\in\mathbb R,
$$ 
while the hypothesis on $f$ implies that 
\begin{equation*}
    \int_{\mathbb R} \|\widehat{f}(t)e^{ita}\|_{\widetilde\A}\d t \leq A\int_{\mathbb R} |\widehat{f}(t)|(1+e^{|t|^{\tau}})\d t=A(\|\widehat f\|_{L^1(\R)}+\|f\|_\tau)\leq 2A\|f\|_\tau , 
\end{equation*} so $f(a)\in \widetilde\A$. If $f(0)=0$, then $0=\int_{\mathbb R} \widehat{f}(t)\d t$ so 
$$
f(a)=f(a)-\frac{1}{2\pi}\int_{\mathbb R} \widehat{f}(t)\d t=\frac{1}{2\pi}\int_{\mathbb R} \widehat{f}(t)(e^{ita}-1)\d t\in\A,
$$ 
since $e^{itb}-1\in\A$. 

The proofs of \emph{(ii)}, \emph{(iii)} and \emph{(iv)} work exactly as in \cite[Theorem 2.5]{Fl24}, so we avoid writing them again. \end{proof}

The above mentioned theorem should be interpreted like this: \emph{(i)} guarantees that the functions of $\mathfrak D_\tau$ act on the elements of $\A$, while \emph{(ii)} states that this action coincides with the continuous functional calculus in any $C^*$-completion of $\A$. \emph{(iii)} is the spectral mapping theorem and \emph{(iv)} is a useful observation.

\begin{defn}
    A reduced Banach $^*$-algebra $\A$ is called {\it locally regular} if there is a subset $R\subset \A_{\rm sa}$, dense in $\A_{\rm sa}$ and such that $\A(b)$ is regular, for all $a\in R$.
\end{defn}

\begin{thm}\label{locallyreg}
    Let $\A$ be a $(k,p,q)$-differential subalgebra of $\B$. Then $\A$ is locally regular.
\end{thm}
\begin{proof}
    We will prove that $\A(b)$ is regular, for all $a\in \A_{\rm sa}$, which is a little stronger than what we need. Indeed, using Proposition \ref{growth}, we get that $\norm{u(ta)}_{\A}=O(e^{{|t|}^\tau})$ as $|t|\to\infty$. This implies that, for some $C>0$, 
    $$
    \int_{\mathbb R}\frac{\log(\norm{u(ta)}_{\widetilde{\A}})}{1+t^2}\d t\leq C\int_{\mathbb R}\frac{|t|^\tau}{1+t^2}\d t<\infty.
    $$ 
    Hence, by a classical criterion of Shilov (see \cite[Example 2.4]{Ne92} for a proof written in English), $\widetilde{\A}(a)=\A(a)$ is regular.
\end{proof}

\subsection{Preservation of spectra}

Let $\Pi:\A\to \mathbb B(\mathcal X)$ be a representation of $\A$ on the Banach space $\mathcal X$. The idea of this subsection is to understand the spectrum of $\Pi(a)$, at least for a self-adjoint $a$. This is particularly important (and also easier) in the case of $^*$-representations, as it allows us to understand the $C^*$-norms in $\A$. We also provide applications to the ideal theory of $\A$. 

If $\A$ is a reduced Banach $^*$-algebra, then $\iota_\A:\A\to {\rm C^*}(\A)$ will denote its canonical embedding into the enveloping $C^*$-algebra ${\rm C^*}(\A)$. The spaces ${\rm Prim}\, {\rm C^*}(\A)$, ${\rm Prim}\,\A$ and ${\rm Prim}_*\A$ denote, respectively, the space of primitive ideals of ${\rm C^*}(\A)$, the space of primitive ideals of $\A$ and the space of kernels of topologically irreducible
$^*$-representations of $\A$, all of them equipped with the Jacobson topology. It is known that $\iota_\A$ induces a continuous surjection ${\rm Prim}\, {\rm C^*}(\A)  \to{\rm Prim}_*\A$ \cite[Corollary
10.5.7]{Pa01}. We recall that for a subset $S\subset \A$, its hull corresponds to 
$$
h(S)=\{I\in {\rm Prim}_*\A\mid S\subset I\},
$$ 
while the kernel of a subset $F\subset {\rm Prim}_*\A$ is 
$$
k(F)=\bigcap\{I\mid I\in F\}.
$$ 

\begin{defn}
    Let $\A$ be a reduced Banach $^*$-algebra. \begin{enumerate}
        \item[(i)] $\A$ is called $C^*$-unique if there is a unique $C^*$-norm on $\A$.
        \item[(ii)] $\A$ is called $^*$-regular if the surjection ${\rm Prim}\, {\rm C^*}(\A)\to {\rm Prim}_*\A$ is a homeomorphism.
    \end{enumerate}
\end{defn} 

It is well-known that $^*$-regularity implies $C^*$-uniqueness. In fact, the latter may be rephrased as the following: $\A$ is $C^*$-unique if and only if, for every closed ideal $I\subset {\rm C^*}(\A)$, one has $I\cap \A\not=\{0\}$. It might also be referred to as the 'ideal intersection property'. The next theorem shows that $^*$-regularity is much stronger.

\begin{thm}\label{regularity}
    Let $\A$ be a $(k,p,q)$-differential subalgebra of $\B$. Then $\A$ is $^*$-regular. In particular, the following are true. \begin{enumerate}
        \item[(i)] For any pair of \,$^*$-representations $\Pi_1, \Pi_2$ of $\A$ the inclusion ${\rm ker}\,\Pi_1\subset {\rm ker}\,\Pi_2$ implies $\norm{\Pi_2(a)}\leq \norm{\Pi_1(a)}$, for all $a\in \A$.
        \item[(ii)] Let $I$ be a closed ideal of $\B$, then $\overline{\A\cap I}^{\norm{\cdot}_{\B}}=I$.
        \item[(iii)] If $I\in {\rm Prim}\, \A$, then $\overline{I}^{\norm{\cdot}_{\B}}\in {\rm Prim}\, \B$. 
    \end{enumerate}
\end{thm}
\begin{proof}
    $^*$-regularity follows from \cite[Theorem 4.2]{Ba81} and \cite[Theorem 10.5.18]{Pa01}. The rest of the assertions are consequences of $^*$-regularity. Indeed, \emph{(i)} follows from \cite[Satz 2]{BLSV78} and \emph{(ii)} from \cite[Theorem 10.5.19]{Pa01}. For \emph{(iii)}, see \cite{Fl24}.
\end{proof}

\begin{rem}
    If $\A$ is a $(k,p,q)$-differential subalgebra of $\B$, then, by the previous theorem, $\A$ must admit a unique $C^*$-norm. This means that 
$$
\|a\|_\B=\|a\|_{{\rm C^*}(\A)},\quad \text{ for all }a\in\A.
$$
In particular, the inclusion $\A\subset \B$ must lift to an isometric $^*$-monomorphism ${\rm C^*}(\A)\to\B$ and 
\begin{equation}\label{rigidity}
    \|{a^{2^k}}\|_{\A}\leq C\norm{a}_{\A}^p\norm{a}_{{\rm C^*}(\A)}^q,\quad \textup{ for all } a\in\A.
\end{equation}
This makes the concept of $(k,p,q)$-differential subalgebra actually very intrinsic to $\A$. For that reason, in the future we might refer to $\A$ as a $(k,p,q)$-differential subalgebra, without any explicit mention to $\B$.
\end{rem}

We finish the subsection with a result that shows that any unital homomorphism preserves the spectral radius. For $C^*$-algebras, this is a classical result of Rickart \cite{Ri53}.
\begin{thm}
    Let $\A$ be a $(k,p,q)$-differential subalgebra. Let $\B$ be a unital Banach algebra and $\varphi:\widetilde{\A}\to\B$ a continuous unital homomorphism. Set $I={\rm ker}\,\varphi$. \begin{enumerate}
        \item[(i)] If $a \in \widetilde{\A}$ is normal, $${\rm Spec}_{\widetilde{\A}/I}(a+I)={\rm Spec}_{\B}\big(\varphi(a)\big).$$ 
        \item[(ii)] For a general $a \in \widetilde{\A}$, $${\rm Spec}_{\widetilde{\A}/I}(a+I)={\rm Spec}_{\B}\big(\varphi(a)\big)\cup \overline{{\rm Spec}_{\B}\big(\varphi(a^*)\big)}.$$
    \end{enumerate}
\end{thm}
\begin{proof}
    This is an application of \cite[Theorem 2.2]{Ba87} and Theorem \ref{symmetry}. $\A$ is $^*$-quotient inverse closed because of local regularity (see \cite[Theorem 2.1]{Ba87}). 
\end{proof}

\subsection{Minimal ideals of a given hull}

Now we turn our attention to the following problem: Let $\A$ be a $(k,p,q)$-differential subalgebra. Show the existence of minimal ideals of $\A$ with a given hull in ${\rm Prim }_*\A$. 

The main application of such a property will be to show the Wiener property, but we postpone that until the next subsection. In order to prove the existence of minimal ideals of a given hull, we follow the approach given in \cite{Lu80}. 

Let $\A$ be a $(k,p,q)$-differential subalgebra. If $F\subset {\rm Prim }_*\A$ is any closed subset, we set 
\begin{equation*}
    \norm{a}_{F}=\sup _{{\rm ker}\,\Pi \in F}\norm{\Pi(a)},\quad \text{ for all }a\in\A,
\end{equation*} 
with the convention that $\norm{a}_\emptyset=0$. Set also
\begin{align*}
    m( F)=&\{f(a) \mid a \in \A_{\rm sa},\norm{a}_{\A} \leq 1, f\in \mathfrak D_\tau, f\equiv 0 \text{ on a neighborhood of }[-\norm{a}_{F},\norm{a}_{F}]\}.
\end{align*}
We let $j(F)$ be the closed two-sided ideal of $\A$ generated by $m(F)$. Note that for $F=\emptyset$, we have
\begin{align*}
m(\emptyset)=&\{f(a) \mid a \in\A_{\rm sa},\norm{a}_{\A} \leq 1, f\in \mathfrak D_\tau, f\equiv 0 \text{ on a neighborhood of }0\}.
\end{align*}

\begin{lem}
    Let $\A$ be $(k,p,q)$-differential subalgebra and $F\subset{\rm Prim }_*\A $ be closed. Then the hull of $j(F)$ is $F$.
\end{lem}

\begin{proof}
    We first prove that $F \subset h(j(F))$. Indeed, let $\Pi$ be a $^*$-representation with ${\rm ker}\,\Pi \in F$ and $f(a) \in m(F)$. Then $\Pi\big(f(a)\big)=f\big(\Pi(a)\big)=0$, since $f\equiv 0$ on the spectrum of $\Pi(a)$. Hence $m(F)\subset j(F) \subset {\rm ker}\, \Pi$ and therefore ${\rm ker}\, \Pi \in h(j(F))$ and $F \subset h(j(F))$.

    On the other hand, let $\Pi$ be a $^*$-representation with ${\rm ker}\,\Pi\in {\rm Prim }_*\A\setminus F$, then, because of Theorem \ref{regularity} there exists $a \in \A$ such that $\norm{a}_{F}<\norm{\Pi(a)}$. In fact, $a$ can be chosen to be self-adjoint and have $\norm{a}_{\A} \leq 1$. In particular, and because of Remark \ref{normal}, we can find a function $f\in \mathfrak D_\tau$ such that $f\equiv 0$ on a neighborhood of $[-\norm{a}_{F},\norm{a}_{F}]$ and $f(\norm{\Pi(a)})\not= 0$. Since $a$ is self-adjoint, $\norm{\Pi(a)}$ lies in the spectrum of $\Pi(a)$ and 
    $$
    \Pi\big(f(a)\big)=f\big(\Pi(a)\big)\not=0,
    $$ 
    from which it follows that ${\rm ker}\,\Pi\not \in h(j({F}))$.
\end{proof}

\begin{thm}\label{minideals}
    Let $\A$ be $(k,p,q)$-differential subalgebra and $F\subset{\rm Prim }_*\A $ be closed. There exists a closed two-sided ideal $j({F})$ of $\A$, with $h(j({F}))={F}$, which is contained in every two-sided closed ideal I with $h(I) \subset {F}$.
\end{thm}

\begin{proof}
    Take a non-zero $f(a) \in m({F})$ and, invoking Remark \ref{normal}, choose $g\in \mathfrak D_\tau$ such that $g\equiv 0$ in $[-\norm{a}_{F},\norm{a}_{F}]$ and $g\equiv 1$ in $f({\rm Spec}_{\A}(a))\cap {\rm Supp}(f)$ (which is compact). Thus 
    $$
    f(a)=(f\cdot g)(a)=f(a)g(a)
    $$ 
    and $g(a)\in m({F})$. This implies that $\A$ satisfies the conditions of \cite[Lemma 2]{Lu80} and the conclusion follows.
\end{proof}

\subsection{The Wiener property}\label{wiener}

We now consider the following property, which is intended as an abstract generalization of Wiener's Tauberian theorem.

\begin{defn}
    Let $\A$ be a Banach $^*$-algebra. We say that $\A$ has the Wiener property $(W)$ if for every proper closed two-sided ideal $I\subset\A$, there exists a topologically irreducible $^*$-representation $\Pi:\A\to\BofH$, such that $I \subset {\rm ker}\,\Pi$. 
\end{defn}

In a moment, we will assume the existence of bounded approximate identities. The next remark is extremely useful if one wants to minimize assumptions.
\begin{rem}\label{onesidedunit}
    Let $\A$ be a Banach $^*$-algebra with a continuous involution. If $\A$ contains a bounded left approximate identity, then it contains a self-adjoint (two-sided) bounded approximate identity. 
    
    Indeed, if $(a_i)_{i\in I}$ is a bounded left approximate identity, then $(a^*_i+a_j-a^*_ia_j)_{i,j\in I^2}$ is a two-sided bounded approximate identity. If $(b_\lambda)_{\lambda\in \Lambda}$ is a two-sided bounded approximate identity, then $(b^*_\lambda b_\lambda)_{\lambda\in \Lambda}$ is a self-adjoint bounded approximate identity. 
\end{rem}

\begin{lem}\label{appidentity}
    Let $\A$ be a $(k,p,q)$-differential subalgebra and let $f\in \mathfrak D_\tau$ such that $f(0)=0$ and $f(1)=1$. Suppose that $(b_{\alpha})_{\alpha}\subset\A$ is a self-adjoint bounded approximate identity. Then 
    $$
    \lim_{\alpha} \,\norm{f(b_{\alpha})a-a}_{\A}=0,
    $$
    for all $a\in \A$.
\end{lem}
\begin{proof}
    Let $B\geq\sup \norm{b_\alpha}_\A $ and without loss of generality, we assume that $B\geq1$. Since $0=f(0)=\int_{\mathbb R}\widehat{f}(t)\d t$ and $1=f(1)=\frac{1}{\sqrt{2\pi}}\int_{\mathbb R}\widehat{f}(t)e^{it}\d t$, one sees that 
    \begin{align*}
    \norm{f(b_{\alpha})a-a}_{\A}&= \frac{1}{\sqrt{2\pi}}\Big\|\int_{\mathbb R}\widehat{f}(t)\big(u(tb_{\alpha})-e^{it}\big)a\, \d t\Big\|_{\A}  \\
    &= \frac{1}{\sqrt{2\pi}}\Big\|{\int_{\mathbb R}\widehat{f}(t)\big(e^{itb_{\alpha}}-e^{it}\big)a\, \d t\Big\|}_{\A}.
\end{align*} 
Let $R>0$ be fixed and let us argue that
\begin{equation}\label{goal2}
    \lim_{\alpha}\Big\|\int_{|t|\leq R}\widehat{f}(t)\big(e^{itb_{\alpha}}-e^{it}\big)a\, \d t\Big\|_{\A} =0.
\end{equation}
Indeed, the desired convergence follows from the computation 
\begin{align*}
    \Big\|\int_{|t|\leq R}\widehat{f}(t)\big(e^{itb_{\alpha}}-e^{it}\big)a\, \d t\Big\|_{\A}&\leq  \|\widehat{f}\|_{C_0(\R)}\int_{|t|\leq R}\big\|\big(e^{itb_{\alpha}}-e^{it}\big)a\big\|_{\A}\d t \\
    &\leq \|\widehat{f}\|_{C_0(\R)}\int_{|t|\leq R}\sum_{k\in\mathbb N} \frac{t^k}{k!}\big\| b_{\alpha}^k a-a\big\|_{\A}\d t \\
    &\leq 2R \|\widehat{f}\|_{C_0(\R)} \sum_{k\in\mathbb N} \frac{R^k}{k!}\|b_{\alpha}^ka-a\|_{\A} \\
    &\leq 2 R \|\widehat{f}\|_{C_0(\R)} \|b_{\alpha}a-a\|_{\A}\sum_{k\in\mathbb N}\frac{R^k}{k!}\sum_{j=0}^{k-1} B^j \\
    &\leq 2 BR^2e^{BR} \|\widehat{f}\|_{C_0(\R)}\|b_{\alpha}a-a\|_{\A}.
\end{align*}
We now turn our attention to the rest of the integral. We see that
\begin{align*}
\Big\|\int_{|t|>R}\widehat{f}(t)\big(e^{itb_{\alpha}}&-e^{it}\big)a\, \d t\Big\|_{\A}\leq  \int_{|t|>R}|\widehat{f}(t)|\|e^{itb_{\alpha}}a \|_{\A}\d t+\|a\|_{\A}\Big|\int_{|t|>R}\widehat{f}(t)e^{it}\, \d t\Big|.
\end{align*}
So it is enough to find $R>0$ such that, given any $\alpha$ and any $\epsilon>0$, one has
\begin{equation}\label{goal}
    \int_{|t|>R}|\widehat{f}(t)|\|e^{itb_{\alpha}}a\|_{\A}\d t<\frac{\epsilon}{2}.
\end{equation}
That is because, in such a case, it is easy to enlarge $R$ to also guarantee
$$
\Big|\int_{|t|>R}\widehat{f}(t)e^{it}\, \d t\Big|<\frac{\epsilon}{2\|a\|_{\A}}
$$
and that, combined with \eqref{goal2}, will yield the result.

Indeed, because of Remark \ref{joingrowth}, one can find a constant $A>0$ such that, independently of $\alpha$, 
$$
\norm{u(tb_\alpha)}_{\A}\leq Ae^{|t|^{\tau}}, \quad\text{ for all }\alpha.
$$
We use this bound to see that
\begin{align*}
    \int_{\R}|\widehat{f}(t)|\|e^{itb_{\alpha}}a \|_{\A}\d t &\leq \|a\|_\A\int_{\R}|\widehat{f}(t)|(1+\|u(tb_{\alpha})\|_{\A})\d t \\
    &\leq \|a\|_\A\|\widehat{f}\|_{L^1(\R)}+A\|a\|_\A\| f\|_{\tau}<\infty.
\end{align*}
Since the integral $\int_{\R}|\widehat{f}(t)|\|e^{itb_{\alpha}}a \|_{\A}\d t$ converges independently of $\alpha$, one can find $R>0$ big enough such that \eqref{goal} holds. This finishes the proof. \end{proof}

\begin{rem}
    If one compares Lemma \ref{appidentity} with its analog in \cite{Fl24}, one will see that our proof here is a lot simpler. The difference, of course, is the fact that the subexponential growth obtained in Proposition \ref{growth} only depends on the norm of the corresponding element $a\in\A$ (Remark \ref{joingrowth}), whereas the polynomial growth estimate obtained in \cite{Fl24} depends on the support and, more critically, on the $L^2$-norm of the corresponding cross-sections. It seems unlikely that the bounded approximate unit fixed in \cite{Fl24} could be bounded in $L^2$-norm, making the application of something like Remark \ref{joingrowth} impossible and resulting in a more involved proof.
\end{rem}

With Lemma \ref{appidentity} at hand, we are finally able to derive the Wiener property for $\A$.

\begin{thm}\label{wienerprop}
    Let $\A$ be a $(k,p,q)$-differential subalgebra that contains a self-adjoint bounded approximate identity. Then $\A$ has the Wiener property.
\end{thm}

\begin{proof}
    By Lemma \ref{appidentity}, $j(\emptyset)$ contains all of $\A$. Now, if $I$ is a closed two-sided ideal of $\A$ such that $h(I)=\emptyset$, then $\A=j(\emptyset) \subset I$ by Theorem \ref{minideals}.
\end{proof}

\section{Intertwining operators and automatic continuity}\label{autcont}

Let $\A$ be a Banach algebra. A Banach space $\mathcal X$ which is also an $\A$-bimodule is called a \emph{Banach $\A$-bimodule} if the maps 
$$
\A\times \mathcal X\ni(a,\xi)\mapsto a\xi \in\mathcal X \quad\text{ and }\quad \X\times\A\ni(\xi,a)\mapsto \xi a \in\mathcal X
$$ 
are jointly continuous. If we only have the continuity of the maps 
$$
\mathcal X\ni\xi\mapsto a\xi \in\mathcal X\quad\text{ and }\quad\mathcal X\ni\xi\mapsto \xi a \in\mathcal X
$$ 
for each $a\in\A$, then $\mathcal X$ is called a \emph{weak Banach $\A$-bimodule}.

\begin{defn}
    Let $\A$ be a Banach algebra and $\X_1,\X_2$ be weak Banach $\A$-bimodules. A linear map $\theta:\X_1\to\X_2$ is called an \emph{$\A$-intertwining operator} if for each $a\in\A$, the maps 
    $$
    \X_1\ni \xi\mapsto \theta(a\xi)-a\theta(\xi)\in \X_2\quad\text{ and }\quad \X_1\ni \xi\mapsto \theta(\xi a)-\theta(\xi)a\in \X_2
    $$ 
    are continuous.
\end{defn}

We mentioned in the introduction that intertwining operators generalize algebra homomorphisms, derivations, and bimodule homomorphisms. We will make this precise in the following example.
\begin{ex}\label{ex-inter}\begin{enumerate}
        \item[(i)] Every $\A$-bimodule homomorphism between weak Banach $\A$-bimodules is an $\A$-intertwining operator.
        \item[(ii)] If $\theta:\B\to\A$ is an algebra homomorphism, then $\A$ can be made into a weak Banach $\B$-bimodules with the actions 
        $$
        ab=a\theta(b)\quad\text{ and }\quad ba=\theta(b)a, \quad\text{for } a\in \A, b\in \B.
        $$ 
        With respect to this actions, $\theta$ is a $\B$-intertwining operator.
        \item[(iii)] Let $\X$ a weak Banach $\A$-bimodule. A derivation is a linear map $D:\A\to \X$ satisfying 
        $$
        D(ab)=D(a)b+aD(b).
        $$ 
        Every derivation is an $\A$-intertwining operator. 
    \end{enumerate}
\end{ex}

\begin{defn}
    Let $\A$ be a Banach algebra, and $\theta:\X_1\to\X_2$ an $\A$-intertwining operator between weak Banach $\A$-bimodules. Then 
    $$
    \mathscr I(\theta)=\{a\in\A\mid \text{the maps } \X_1\ni \xi\mapsto \theta(a\xi)\in \X_2,\,\X_1\ni \xi\mapsto \theta(\xi a)\in \X_2 \text{ are continuous } \forall a\in\A\}
    $$ 
    is the \emph{continuity ideal} of $\theta$.
\end{defn}

Note that $\mathscr I(\theta)$ is not necessarily closed, unless $\X_2$ is a Banach $\A$-bimodule. In order to introduce the next theorem, we now recall the notion of a set of synthesis.

\begin{defn}
    A closed subset $C\subset {\rm Prim}_*\A$ will be called a \emph{set of synthesis} if $k(C)$ is the unique closed two-sided ideal $I\subset \A$ such that $h(I)=C$.
\end{defn}

The following theorem was proven in \cite{Fl25} as a generalization of \cite[Theorem 2.3, Theorem 3.2]{Ru96}. The reader should note that we have already shown that our algebras automatically satisfy most of the assumptions, so the theorem can be applied with great generality. Any Banach algebra $\A$ will be considered a bimodule over itself with left/right multiplication, as described in Example \ref{ex-inter}.

\begin{thm}[\cite{Fl25}]\label{theprop}
    Let $\A$ be a reduced Banach $^*$-algebra, $\X_1$ a Banach $\A$-bimodule, $\X_2$ a weak Banach $\A$-bimodule and $\theta:\X_1\to\X_2$ an $\A$-intertwining operator. Further suppose that 
    \begin{enumerate}
        \item[(i)] For all $a\in\A_{\rm sa}$, $\A(a)$ is a regular Banach algebra.
        \item[(ii)] $\A$ has the Wiener property $(W)$.
        \item[(iii)] Every finite subset $F\subset{\rm Prim}_*\A$ such that all $P\in F$ has finite codimension is a set of synthesis for $\A$.
    \end{enumerate} Then $\overline{\mathscr I(\theta)}^{\norm{\cdot}_{\A}}$ has finite codimension in $\A$. Furthermore, if $\X_1=\A$ and $\A$ also satisfies \begin{enumerate}
        \item[(iv)] Every closed two-sided $^*$-ideal $I\subset \A$ of finite codimension has a bounded left approximate identity.
    \end{enumerate} 
    Then $\theta$ is continuous if and only if $\mathscr I(\theta)$ is closed. \end{thm}
    \begin{proof}
        This is exactly \cite[Theorem 3.6]{Fl25}, except for conditions \emph{(i)} and \emph{(iv)}. Condition \emph{(i)} is obviously stronger now, so let us explain why the weakening of Condition \emph{(iv)} still implies the same result. Indeed, following the proof of \cite[Theorem 3.6]{Fl25}, one sees that 
        $$ \overline{\mathscr I(\theta)}^{\norm{\cdot}_{\A}}=\bigcap_{\substack{P\supset \mathscr I(\theta) \\ P\in {\rm Prim}_*\A} }P=k\big(\mathscr I(\theta) \big),$$
        which shows that $\overline{\mathscr I(\theta)}^{\norm{\cdot}_{\A}}$ is a $^*$-ideal. Now, we note that condition \emph{(iv)} was only required to ensure that $\overline{\mathscr I(\theta)}^{\norm{\cdot}_{\A}}$ contains a left approximate identity.
    \end{proof}

As in \cite{Fl25}, condition \emph{(iv)} in Theorem \ref{theprop} is the most restrictive. In fact, we do not know of any simple way to imply it with generality. It is satisfied by $C^*$-algebras and amenable Banach algebras \cite[Proposition VII.2.31]{He89}. The amenability condition, however, seems way too restrictive. 

Condition \emph{(iv)} can obviously be bypassed if nontrivial ideals of finite codimension do not exist inside $\A$. The next result aims in that direction.

\begin{thm}[\cite{Fl25}]\label{theprop-simple}
    Let $\A$ be a symmetric, reduced, unital Banach $^*$-algebra, $\X_1$ a Banach $\A$-bimodule, $\X_2$ a weak Banach $\A$-bimodule and $\theta:\X_1\to\X_2$ an $\A$-intertwining operator. Further suppose that 
    \begin{enumerate}
        \item[(i)] For all $a\in\A_{\rm sa}$, $\A(a)$ is a regular Banach algebra.
        \item[(ii)] $\A$ has the Wiener property $(W)$.
        \item[(iii)] ${\rm C^*}(\A)$ has no proper closed two-sided ideals with finite codimension.
    \end{enumerate} Then $\theta$ is continuous.
\end{thm}

Note that conditions \emph{(i)} and \emph{(ii)} were already proven for $(k,p,q)$-differential subalgebras (see Theorem \ref{locallyreg} and Theorem \ref{wienerprop}). So the next theorem is an immediate application of Theorem \ref{theprop-simple}.

\begin{thm}\label{theprop-simple-applied}
    Let $\A$ be a unital $(k,p,q)$-differential subalgebra. Also consider a Banach $\A$-bimodule $\X_1$, a weak Banach $\A$-bimodule $\X_2$, and an $\A$-intertwining operator $\theta:\X_1\to\X_2$. If ${\rm C^*}(\A)$ has no proper closed two-sided ideals with finite codimension, then $\theta$ is continuous.
\end{thm}

In order to also apply Theorem \ref{theprop} in our setting, we still need to verify its condition \emph{(iii)}. We will do that now. We follow an approach inspired by \cite[Lemma 10]{Ba83}.

\begin{lem}\label{barnes}
    Let $\A$ be a $(k,p,q)$-differential subalgebra. Let $I$ be a closed two-sided $^*$-ideal of $\A$ that contains a bounded left approximate identity. If $a\in I$ and $\epsilon>0$, then there exist $b_1,b_2\in I$ such that 
    $$
    \norm{a-b_1a}_{\A}<\epsilon\quad\text{ and }\quad b_2b_1=b_1. 
    $$
\end{lem}

\begin{proof}
    Without loss of generality, we may assume that $I$ contains a self-adjoint bounded approximate identity $b_\alpha\in I$ (see Remark \ref{onesidedunit}).

    Let $f\in \mathfrak D_\tau$ such that $f(0)=0$ and $f(1)=1$. Note that $f(b_\alpha)\in I$ and, repeating the proof of Lemma \ref{appidentity}, one gets that 
    $$
    \lim_{\alpha} \| f(b_\alpha)a-a\|_\A=0, \quad \text{ for all }a\in I.
    $$
    If we further suppose that $f$ is compactly supported and pick $g\in \mathfrak D_\tau$ such that $g\equiv 1$ on ${\rm Supp}(f)$, then we also have
    $$
    g(b_\alpha)f(b_\alpha)=(g\cdot f)(b_\alpha)=f(b_\alpha).
    $$
    Choosing $\alpha$ large enough proves the claim.
\end{proof}

\begin{thm}\label{synth}
    Let $\A$ be a $(k,p,q)$-differential subalgebra and suppose that every closed two-sided $^*$-ideal $I\subset\A$ of finite codimension has a bounded left approximate identity. Let $F$ be a finite closed subset of $\,{\rm Prim}_*\A$ such that every $P\in F$ has finite codimension. Then $F$ is a set of synthesis.
\end{thm}

\begin{proof}
    Define 
    $$
    M=\{b\in k(F)\mid \exists a \in k(F) \text{ such that }ab=b\}.
    $$
    One obviously has that $h(M)\supset F$, but Lemma \ref{barnes} actually implies $h(M)= F$. Indeed, $k(F)=\cap_{P\in F}P$ is a finite codimensional, closed two-sided $^*$-ideal, so for every $a\in k(F)$ there is a sequence $b_n\in M$ making $a=\lim_{n\to\infty} b_na$. Therefore $a\in P'$, for all $P'\in h(M)$ and $h(M)= F$. 
    
    It clear that $F\subset h(a)$ for all the elements $a$ involved in the definition of $M$. Now, if $I$ is a closed two-sided ideal of $\A$ with $h(I)=F$, then because of \cite[Lemma 2]{Lu80}, we have $M\subset I$. Then applying Lemma \ref{barnes} again, we have $k(F)\subset I$ and hence $I=k(F)$.
\end{proof}

With this, we can state our general criterion for automatic continuity. 

\begin{thm}\label{theprop-applied}
    Let $\A$ be a $(k,p,q)$-differential subalgebra where all closed $^*$-ideals of finite codimension have a bounded left approximate identity. Let $\X$ be a weak Banach $\A$-bimodule and $\theta:\A\to\X$ an $\A$-intertwining operator. Then $\theta$ is continuous if and only if $\mathscr I(\theta)$ is closed. 
\end{thm}
\begin{proof}
    It follows from Theorem \ref{theprop}. \emph{(i)} and \emph{(ii)} were proven in Theorem \ref{locallyreg} and Theorem \ref{wienerprop}. Condition \emph{(iii)} is covered by Theorem \ref{synth} and condition \emph{(iv)} is taken as an assumption.
\end{proof}

In the setting of the above theorem, if $\X$ is a Banach $\A$-bimodule instead of a weak Banach $\A$-bimodule, then $\mathscr I(\theta)$ is closed and we obtain the (automatic) continuity of $\theta$. This is recorded in the next corollaries.

\begin{cor}\label{thecor-applied}
    Let $\A$ be a $(k,p,q)$-differential subalgebra where all closed $^*$-ideals of finite codimension have a bounded left approximate identity. Let $\X$ be a Banach $\A$-bimodule and $\theta:\A\to\X$ an $\A$-intertwining operator. Then $\theta$ is continuous. 
\end{cor}

\begin{cor}\label{thecor-amenable}
    Let $\A$ be an amenable $(k,p,q)$-differential subalgebra. Let $\X$ be a Banach $\A$-bimodule and $\theta:\A\to\X$ an $\A$-intertwining operator. Then $\theta$ is continuous. 
\end{cor}

Finally, we also consider the specific case of derivations. In such a case, we can appeal to the recent results in \cite{flores2025der} to obtain the following two theorems. They show that finding dense $(k,p,q)$-differential subalgebras can be a useful tool to obtain results about automatic continuity for the ambient algebra.

\begin{thm}
    Let $\B$ be a Banach algebra that contains a $(k,p,q)$-differential subalgebra $\A$ as a Banach subalgebra. Let $D:\B\to\X$ be a derivation into a Banach $\B$-bimodule. Further suppose that $\A$ is unital and that ${\rm C^*}(\A)$ has no proper closed two-sided ideals of finite codimension. Then $D$ is continuous.
\end{thm}

\begin{thm}
Let $\A,\B,\mathfrak C$ be Banach algebras such that $\A$ is a $(k,p,q)$-differential subalgebra and such that the inclusions $\A\subset\B,\mathfrak C\subset {\rm C^*}(\A)$ are continuous. Then every derivation $D:\B\to \mathfrak C$ is continuous.
\end{thm}

When applied to a differential subalgebra, the above theorem recovers \cite[Theorem 13.1(ii)]{KS94}, and, in the particular case of domains of closed $^*$-derivations (see Subsection \ref{derivations}), it recovers Longo's relative boundedness result \cite{Lo79}.

\section*{Acknowledgments}

The author thankfully acknowledges support by the NSF grant DMS-2144739. He is also grateful to Professor Ben Hayes for the interesting discussions surrounding the present topic.

\printbibliography

\bigskip
\bigskip
ADDRESS

\smallskip
\smallskip
Felipe I. Flores

Department of Mathematics, University of Virginia,

114 Kerchof Hall. 141 Cabell Dr,

Charlottesville, Virginia, United States

E-mail: hmy3tf@virginia.edu
\end{document}